\documentclass[a4paper,11pt,reqno]{amsart}
\usepackage{latexsym,amsmath,amssymb,amsfonts}
\usepackage{a4wide}
\usepackage{verbatim}
\usepackage{tikz}
\usetikzlibrary{patterns}
\usepackage[bookmarks]{hyperref}
\usepackage[deletedmarkup=xout]{changes}

\theoremstyle{plain}
\begingroup
\newtheorem{theorem}{Theorem}[section]
\newtheorem{lemma}[theorem]{Lemma}
\newtheorem{proposition}[theorem]{Proposition}
\newtheorem{corollary}[theorem]{Corollary}

\endgroup
\theoremstyle{definition}
\begingroup

\endgroup
\theoremstyle{remark}
\newtheorem{remark}[theorem]{Remark}

\numberwithin{equation}{section}

\newcommand{\EE}{{\mathcal E}}

\newcommand{\LL}{{\mathcal L}}

\newcommand{\lam}{\lambda}
\renewcommand{\phi}{\varphi}

\newcommand{\Gam}{\Gamma}

\newcommand{\Sig}{\Sigma}
\newcommand{\Ome}{\Omega}

\newcommand{\e}{\varepsilon}
\newcommand{\vphi}{\varphi}

\newcommand{\R}{\mathbb R}

\newcommand{\Ha}{{\mathcal H}}
\newcommand{\Chi}{{\mathcal X}}
\newcommand{\wto}{\rightharpoonup}
\newcommand{\wstar}{\stackrel{*}{\rightharpoonup}}
\newcommand{\dd}{\mathrm{d}}
\newcommand{\dds}{\ \mathrm{d}}
\newcommand{\p}{\partial}
\renewcommand{\t}{\tilde}

\newcommand{\dist}{{\rm dist\,}}
\newcommand{\diam}{{\rm diam\,}}
\newcommand{\supp}{{\rm supp\,}}
\newcommand{\trace}{{\rm Tr\,}}
\newcommand{\mindiam}{{\rm min\,diam\,}}

\DeclareMathOperator*{\loc}{loc}

\DeclareMathOperator{\graph}{graph}
\newcommand{\qqquad}[0]{\qquad\qquad}

\newcommand{\BS}[0]{\backslash}

\newcommand{\Mattes}{\mathbin{\vrule height 1.6ex depth 0pt width 0.13ex\vrule height 0.13ex depth 0pt width 1.3ex}}

 \newcommand{\lupref}[2]{\hspace{0ex}
  \stackrel{\eqref{#1}}{#2}} % in line
\newcommand{\average}{{\mathchoice {\kern1ex\vcenter{\hrule height.4pt width 6pt depth0pt} \kern-9.7pt}
    {\kern1ex\vcenter{\hrule height.4pt width 4.3pt depth0pt} \kern-7pt} {} {} }} \newcommand{\med}{\average\int}

\title[Optimal distribution of oppositely charged phases]{Optimal distribution of oppositely charged phases: perfect screening and other properties}
\author{Marco Bonacini \and Hans Kn\"upfer \and Matthias R\"{o}ger}
\date{\today}

\address[Marco Bonacini]{Institut f\"ur Angewandte Mathematik, Universit\"at Heidelberg, Im Neuenheimer Feld 294, 69120 Heidelberg, Germany}
\email{marco.bonacini@uni-heidelberg.de}
\address[Hans Kn\"upfer]{Institut f\"ur Angewandte Mathematik and Interdisciplinary Center for Scientific Computing (IWR), Universit\"at Heidelberg, 69120 Heidelberg, Germany}
\email{knuepfer@uni-heidelberg.de}
\address[Matthias R\"{o}ger]{Fakult\"{a}t f\"{u}r Mathematik, Technische Universit\"{a}t Dortmund, Vogelpothsweg 87, D-44227 Dortmund, Germany}
\email{matthias.roeger@tu-dortmund.de}


\begin{document}

\begin{abstract}
  We study the minimum energy configuration of a uniform distribution of negative charge subject to Coulomb repulsive self-interaction and attractive interaction with a fixed positively charged domain. After having established existence and uniqueness of a minimizing configuration, we prove charge neutrality and the complete \textit{screening} of the Coulomb potential exerted by the positive charge, and we discuss the regularity properties of the solution. We also determine, in the variational sense of $\Gamma$-convergence, the limit model when the charge density of the negative phase is much higher than the positive one.
\end{abstract}

\maketitle

%%%%%%%%%%%%%%%%%%%%%%%%%%%%%%%%%%%%%%%%%%%%%%%%%%%%%%%%%%%%%%%%%%%%%%%%%%%%%%%%%%%%%%%%%%%%%%%%%%%%%%%%%
%%%%%%%%%%%%%%%%%%%%%%%%%%%%%%%%%%%%%%%%%%%%%%%%%%%%%%%%%%%%%%%%%%%%%%%%%%%%%%%%%%%%%%%%%%%%%%%%%%%%%%%%%

\section{Introduction} \label{sec-intro} %

In this paper, we investigate ground states of the energy for a system including both attractive as well as repulsive Coulomb interactions. The very fundamental nature of such \textit{nonlocal} Coulomb interaction is testified by its ubiquitous presence in nature and by the vast number of its occurrences in physical systems.

We consider the problem of finding the optimal shape taken by a uniform negative distribution of charge interacting with a fixed positively charged region; mathematically, this leads to a problem in \textit{potential theory} (\cite{Hel,Lan}). In our model, minimizing configurations are determined by the interplay between the repulsive self-interaction of the negative phase and the attractive interaction between the two oppositely charged regions. We investigate existence of global minimizers and their structure, and we obtain \textit{charge neutrality} and \textit{screening} as key features of our system. In particular, it is noteworthy that the positive phase is completely screened by the optimal negative distribution of charge, in the sense that outside the support of the two charges the long-range potential exerted by the positive region is canceled by the presence of the negative one.

On the basis of this screening result, we can draw a link to the classical theory of \textit{obstacle problems} \cite{Caf77,Caf80,Caf98,PetShaUra}: indeed, the net potential of the optimal configuration can be characterized -- outside the positively charged region and with respect to its own boundary conditions -- as the solution to an obstacle problem, a fact which in turn entails further regularity properties of the minimizer.

\medskip

Mathematically, we represent the fixed positively charged domain by a bounded open set $\Omega^+\subset\R^3$, and we are
interested in minimizing among configurations $\Omega^-\subset\R^3\setminus\Omega^+$ with finite volume the
nonlocal energy
\begin{align} \label{def-E-intro} 
	&I(\Omega^+,\Ome^-)\ := \ \notag\\
	  &\int_{\Omega^+}\int_{\Omega^+}\frac{1}{4\pi|x-y|} \dds x \dd y +
  \int_{\Omega^-}\int_{\Omega^-}\frac{1}{4\pi|x-y|}\dds x \dd y -2 \int_{\Omega^+}\int_{\Omega^-}\frac{1}{4\pi|x-y|}\dds x \dd y\,.
\end{align}
Here the first two terms represent the repulsive self-interaction energies of $\Omega^+$ and $\Omega^-$, respectively, and the third term represents the
attractive mutual interaction between $\Omega^+$ and $\Omega^-$. The present model also arises in the modelling of
copolymer-homopolymer blends, see the remarks on related models below.

The first natural question concerns the existence of minimizers for this variational problem. Since the functional does
not include any interfacial penalization, and we just have a uniform bound on the charge density, the natural topology for the compactness of minimizing sequences is the
weak*-topology in $L^\infty(\R^3)$. While the lower semicontinuity of the functional follows from standard arguments in potential
theory, a non-trivial issue lies in the fact that the limit distribution could include intermediate densities of charge,
in the sense that the limit function could attain values in the whole interval $[0,1]$: as a result, the limit
configuration might not be admissible for our problem. We will however show below that minimizer only take values in $\{0,1\}$, which allows to bring the negatively charged region and the positively charged one as ``close'' together as possible.

Our next aim is to identify specific properties of the optimal set. Here we first establish a \textit{charge neutrality} phenomenon: the total negative charge of the optimal configuration equals the given positive one, i.e. $|\Omega^-|=|\Omega^+|$. In particular, this shows that configurations with nonzero total net charge are unstable in this sense. We also study the case where the total negative charge is prescribed: we discuss the minimization of the energy under the additional volume constraint $|\Omega^-|=\lambda$, and analyze the dependence of the solution on the parameter $\lambda$ (Theorem~\ref{thm-constrained}), proving that a minimizer exists if and only if $\lambda\leq|\Omega^+|$.
The issue of charge neutrality is a central question for systems including interacting positive and negative charges. For instance, we refer to the work of Lieb and Simon \cite{LieSim}, where charge neutrality is shown for minimizers of the Thomas-Fermi energy functional for atomic structures, in the context of quantum mechanics.
Another related question is whether the maximal negative \textit{ionization} (the number of extra electrons that a neutral atom can bind) remains small: we mention in particular the so-called \textit{ionization conjecture}, which gives an upper bound on the number of electrons that can be bound to an atomic nucleus. For some results in this direction, see e.g. \cite{BenLie,Lieb84,LiSiSiTh,Sol03}.

A second remarkable property of minimizers is that complete \textit{screening} is achieved (Theorem~\ref{thm-screening}): the
negative charge tends to arrange itself in a layer around the boundary of $\Omega^+$ in such a way to cancel the Coulomb
potential exerted by the positive charge. Indeed, the net potential 
\begin{align*} 
  \phi(x):= \int_{\Omega^+}\frac{1}{4\pi|x-y|}\dds y - \int_{\Omega^-}\frac{1}{4\pi|x-y|}\dds y
\end{align*}
of the optimal configuration
actually vanishes in the external uncharged region $\R^3\setminus(\overline{\Omega^+\cup\Omega^-})$. Since it can be
proved (under some mild regularity assumptions on the boundary of $\Omega^+$) that $\phi$ is strictly positive in the closure of $\Omega^+$, this implies that the positive phase is completely
surrounded by $\Omega^-$ and the distance between $\Omega^+$ and the uncharged space is strictly positive; moreover,
each connected component of $\Omega^-$ has to touch the boundary of $\Omega^+$. Notice that, although the minimizer $\Omega^-$ is in general defined up to a Lebesgue-negligible set, we can always select a precise representative, which is in particular an open set, see \eqref{def-Ome-}. Such properties are established by combining information from the Euler-Lagrange equations with \textit{ad hoc} arguments based on the maximum principle.

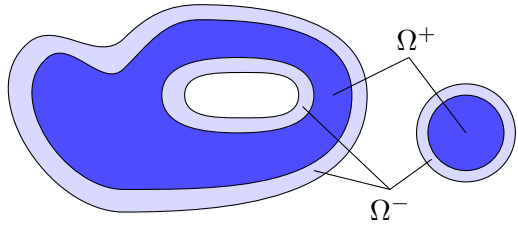
\begin{figure}
\begin{tikzpicture}[scale=0.5]
\draw [fill=blue!15!white] (-2.5,2.4) to [out=40, in=225] (0,2.5) to [out=45, in=180 ] (2,3.4) to [out=0, in=90] (6.4,1) to [out=270, in=0] (0,-2.1) to [out=180, in=220] (-2.5,2.4);
\draw [fill=blue!70!white] (-2,2) to [out=40, in=225] (0,1.7) to [out=45, in=180 ] (2,3) to [out=0, in=90] (6,1) to [out=270, in=0] (0,-1.5) to [out=180, in=220] (-2,2);
\path [fill=white] (1,1) to [out=90, in=180] (3,2) to [out=0, in=90] (5,1) to [out=270, in=0] (3,0) to [out=180, in=270] (1,1);
\draw [fill=blue!15!white] (1,1) to [out=90, in=180] (3,2) to [out=0, in=90] (5,1) to [out=270, in=0] (3,0) to [out=180, in=270] (1,1);
\draw [fill=white] (1.6,1) to [out=90, in=180] (3,1.6) to [out=0, in=90] (4.6,1) to [out=270, in=0] (3,0.4) to [out=180, in=270] (1.6,1);
\draw [fill=blue!15!white] (9,1.3) to [out=0, in=90] (10.3,0) to [out=270, in=0] (9,-1.3) to [out=180, in=270] (7.7,0) to [out=90, in=180] (9,1.3);
\draw [fill=blue!70!white] (9,1) to [out=0, in=90] (10,0) to [out=270, in=0] (9,-1) to [out=180, in=270] (8,0) to [out=90, in=180] (9,1);
\node at (7.7,2.5) {$\Omega^+$};
\draw (7.5,2) to (5.5,1);
\draw (7.5,2) to (9,0);
\node at (7,-2) {$\Omega^-$};
\draw (7,-1.5) to (8.1,-0.7);
\draw (7,-1.5) to (4.7,0.7);
\draw (7,-1.5) to (5,-1);
\end{tikzpicture}
\caption{Sketch of the shape of the minimizer $\Omega^-$ (light-shaded area) corresponding to a given configuration $\Omega^+$ (dark-shaded area): the negative charges arrange in a layer around the positively charged region.}
\label{fig-intro}
\end{figure}

The screening result enables us to draw a connection between the classical theory of the obstacle problem and our
model, showing that the potential $\phi$ of a minimizing configuration is actually a solution of the former. In turn,
this allows us to exploit the regularity theory for the \textit{free boundary} of solutions to obstacle-type problems in order to recover further regularity properties of the minimizing set (Theorem~\ref{thm-regularity}).

In the last part of the paper, we investigate the regime in which the charge density of the negative phase is much
higher than the positive one, which is modeled mathematically by rescaling the negative charge density by $\frac{1}{\e}$
and letting $\e\to 0$. In this case, we prove $\Gamma$-convergence to a limit model where the distribution of the
negative charge is described by a positive Radon measure (Theorem~\ref{thm-gammaconv}); in turn, we show that the
optimal configuration for this limit model is attained by a surface distribution of charge on $\partial\Omega^+$
(Proposition~\ref{prp-surfacecharge}).

\medskip

\textbf{Related models with Coulomb interaction.}
Capet and Friesecke investigated in \cite{CapFri} a closely related discrete model, where the optimal
distribution of $N$ electrons of charge -1 in the potential of some fixed positively charged atomic nuclei is determined
in the large $N$ limit. Under a hard-core radius constraint, which prevents electrons from falling into the nucleus,
they show via $\Gamma$-convergence that the negative charges tend to uniformly distribute on spheres around the atomic
nuclei, the number of electrons surrounding each nucleus matching the total nuclear charge; in particular, the potential
exerted by the nuclei is screened and in the limit the monopole moment, the higher multipole moments of each atom, and
the interaction energy between atoms vanish. Hence our analysis on charge neutrality, screening and on the limit surface charge model could also be interpreted as a macroscopic counterpart
of the discrete analysis developed in \cite{CapFri}.

Due to the universal nature of Coulomb interaction, we expect that our results could be also instrumental in the investigation of more general models, where an interfacial penalization is possibly added and the phase $\Omega^+$ is no longer fixed.
Recently, the problem with a single self-interacting phase of prescribed volume, surrounded by a neutral phase, has been extensively studied (see e.g. \cite{BonCri,FraLie,Jul,KnuMur1,KnuMur2,LuOtto}). 

A minimization problem for two phases with Coulomb interactions and an interfacial energy term arises for instance in the modeling of diblock-copolymers.  These consist of two subchains of different type that repel each other but are chemically bonded, leading to a phase separation on a mesoscopic scale. Variational models derived by mean
field or density functional theory 
\cite{MaSc94,Leib80,OhKa86,BaOo90,ChRe03} take the form of a nonlocal Cahn--Hilliard energy. A subsequent strong segregation limit results in a nonlocal perimeter problem \cite{ReWe00} with a Coulomb-type energy contribution. For a mathematical analysis of diblock-copolymer models see for example \cite{AlCo09,ReWe08,ChPW09}.

In a mixture of diblock-copolymers and homopolymers an additional \textit{macroscopic
phase separation} in homopolymers and diblock-copolymers occurs, where now three phases have to be distinguished. Choksy and Ren developed a density functional theory
\cite{ChRe05}, a subsequent strong segregation reduction leads to an energy of the form
\begin{align}
   c_0\Ha^2(\partial (\overline{\Omega^+\cup \Omega^-})) + c_1 \Ha^2(\partial \Omega^+) + c_2 \Ha^2(\partial \Omega^-) + I(\Omega^+,\Omega^-) \label{eq-energy-bcp}
\end{align}
where $\Omega^+,\Omega^-$ are constrained to be open sets of finite perimeter, to have disjoint supports and equal
total volume, and where $I(\Omega^+,\Omega^-)$ denotes the Coulomb interaction energy defined in \eqref{def-E-intro}. In \cite{GenPel} the existence of minimizers
has been shown in one space dimension and lower and upper bounds on the energy of minimizers have been presented in higher dimensions. Furthermore, in \cite{GePe09} the stability of layered structures has been investigated.

The model that we analyze in the present paper can be understood as a reduction of \eqref{eq-energy-bcp} to the case $c_0=c_2=0$ and a
minimization in $\Omega^-$ only, for $\Omega^+$ given.

\medskip

\textbf{Structure of paper.}  The paper is organized as follows. The notation and the variational setting of the problem
are fixed in Section~\ref{sec-setting}, where we also state the main results. A detailed discussion of the
relaxed model with intermediate densities of charge, instrumental for the analysis of the original problem, is performed
in Section~\ref{sec-relaxed}, where the main existence theorems are proved. Section~\ref{sec-screening} contains the
proof of the screening property, while in Section~\ref{sec-obstacle} the relation with the obstacle problem and its
consequences are discussed (in particular, we prove the regularity of the minimizer). Section~\ref{sec-surface} is
devoted to the analysis of a limit surface-charge model. Finally, spherically symmetric configurations are explicitly
discussed in the concluding Appendix.

\medskip

\textbf{Notation.} We denote the ball centered at a point $x\in\R^3$ with radius $\rho>0$ by $B_\rho(x)$,
writing for simplicity $B_\rho$ for balls centered at the origin. For any measurable set $E \subset \R^3$, we denote its
Lebesgue measure by $|E| := \LL^3(E)$. The integral average of an integrable function $f$ over a measurable set $E$ with
positive measure is $\med_E f := \frac{1}{|E|}\int_E f$. Sublevel sets of a function $f$ are indicated by
$\{f<\alpha\}:=\{x\in\R^3 : f(x) < \alpha\}$, and a similar notation is used for level sets and superlevel sets.

%%%%%%%%%%%%%%%%%%%%%%%%%%%%%%%%%%%%%%%%%%%%%%%%%%%%%%%%%%%%%%%%%%%%%%%%%%%%%%%%%%%%%%%%%%%%%%%%%%%%%%%%%
%%%%%%%%%%%%%%%%%%%%%%%%%%%%%%%%%%%%%%%%%%%%%%%%%%%%%%%%%%%%%%%%%%%%%%%%%%%%%%%%%%%%%%%%%%%%%%%%%%%%%%%%%
  
\section{Setting and main results} \label{sec-setting} %

Let $\Omega^+\subset\R^3$ be a fixed non-empty, bounded and open set. We assume that $\Ome^+$ is uniformly, positively charged with charge density $1$. For any uniformly, negatively charged measurable set $\Ome^-$ with finite Lebesgue measure, we consider the corresponding Coulombic energy $E(\Omega^-):=I(\Omega^+,\Omega^-)$, thus
\begin{align} \label{def-E} 
	E(\Ome^-)= \int_{\Omega^+}\int_{\Omega^+}\frac{1}{4\pi|x-y|} \dds x \dd y +
  \int_{\Omega^-}\int_{\Omega^-}\frac{1}{4\pi|x-y|}\dds x \dd y -2 \int_{\Omega^+}\int_{\Omega^-}\frac{1}{4\pi|x-y|}\dds x \dd y\,.
\end{align}

Our aim is to find the optimal configuration of the negative charge,
under the assumption that the two oppositely charged regions do not overlap. We hence consider the minimization problem
\begin{align} \label{min-unc} %
  \min \ \Bigl\{ E(\Ome^-) \ :\ \Omega^-\subset\R^3 \text{ measurable, } |\Omega^+\cap\Omega^-|=0,\ |\Omega^-|<\infty \Bigr\}
\end{align}
(notice that we require that $\Omega^-$ has finite volume in order for the energy \eqref{def-E} to be well defined). We also consider the closely related minimization problem where the total negative charge is prescribed, which for $\lambda>0$ given yields
\begin{align} \label{min-con} %
  \min \ \Bigl\{ E(\Ome^-) \ :\ \Omega^-\subset\R^3\text{ measurable, } |\Omega^+\cap\Omega^-|=0,\ |\Ome^-| = \lam \ \Bigr\}\,.
\end{align}

\medskip

The energy \eqref{def-E} can be expressed in different ways. We will usually denote $u^+:=\chi_{\Omega^+}$ and
$u:=\chi_{\Omega^-}$, where $\chi_{\Ome^\pm}$ are the characteristic functions of the sets $\Ome^\pm$. For given charge
densities $u^+,u$, the associated \textit{potential} $\phi$ is defined as
\begin{align} \label{def-phi} %
  \phi(x) := \int_{\R^3} \frac{u^+(y)-u(y)}{4\pi|x-y|} \dds y\,.
\end{align}
Notice that the potential $\phi$ solves the elliptic problem
\begin{align*}
  \begin{cases}
    -\Delta\phi = u^+-u,\\
    \lim_{|x|\to\infty} |\phi(x)| =0
  \end{cases}
\end{align*}
(see Lemma~\ref{lem-decay} for the second condition). By classical elliptic regularity, we have $\phi\in
C^{1,\alpha}(\R^3)$ for every $\alpha<1$ and $\phi\in W^{2,p}_{\loc}(\R^3)$ for all $1\leq p<\infty$. In addition, $\phi\in L^p(\R^3)$  for all $p>3$, and $\nabla\phi\in L^q(\R^3)$ for all $q>\frac{3}{2}$ by \cite[Theorem~4.3, Theorem~10.2]{LiLo01}. A standard argument, based on integration by parts, shows that the energy of a
configuration can be expressed in terms of the associated potential as
\begin{align*}
  E(\Ome^-) = \int_{\R^3}|\nabla\phi|^2\dds x\,.
\end{align*}
Finally, yet another way to represent the energy is in terms of Sobolev norms: indeed,
\begin{align*} 
  E(\Ome^-) = \|u^+ - u\|^2_{H^{-1}(\R^3)}\,.
\end{align*}

\medskip

We now state the main findings of our analysis. We first consider the unconstrained minimization problem \eqref{min-unc}, in which the total negative charge is not \textit{a priori} prescribed, proving existence and uniqueness of a minimizing configuration; interestingly, it turns out that the volume of the minimizer matches the volume of the positive charge and the system exhibits a charge neutrality phenomenon.

\begin{theorem}[The unconstrained problem: Existence and uniqueness]\label{thm-unconstrained} %
  Let $\Omega^+\subset\R^3$ be a fixed, non-empty, bounded and open set. Then, the minimum problem \eqref{min-unc} admits a unique (up to a set of zero Lebesgue measure) solution
  $\Ome^- \subset \R^3$. Furthermore, the minimizer satisfies the saturation property
  \begin{align*}
    |\Omega^-|= |\Ome^+|\,.
  \end{align*}
\end{theorem}
We also obtain a corresponding result for the case \eqref{min-con} when the total negative charge is prescribed.
\begin{theorem}[The constrained problem: Existence and uniqueness/Nonexistence]\label{thm-constrained} %
  Let $\Omega^+\subset\R^3$ be a fixed, non-empty, bounded and open set. Then:
  \begin{enumerate}
  \item For every $\lam \leq |\Ome^+|$, there is a unique (up to a set of zero Lebesgue measure) minimizer $\Ome^-$ of
    \eqref{min-con}.
  \item For every $\lam > |\Ome^+|$, there is no global minimizer of \eqref{min-con}.
  \end{enumerate}
\end{theorem}
We will first prove Theorem~\ref{thm-constrained}, then Theorem~\ref{thm-unconstrained} is an easy consequence of this
theorem.  The main technical difficulty arising when we try to apply the direct method of the Calculus of Variations to
prove the existence of a minimizer is clearly the following: for a minimizing sequence $(\Omega^-_n)_n$, we can pass to
a subsequence such that $\chi_{\Omega^-_n}\wto u$ weakly* in $L^\infty(\R^3)$, but we can not guarantee that $u$ takes
values in $\{0,1\}$: in other words, the limit object might no longer be a set, with a uniform distribution of
charge. This obstacle will be bypassed by considering the relaxed problem where we allow for intermediate densities of
charge, and showing that a minimizer of this auxiliary problem is in fact a minimizer of the original one (see
Section~\ref{sec-relaxed} for the proofs of these results).
We remark that a similar strategy was also used in \cite{GenPel} for a related one-dimensional model.

\medskip

Having established existence of a solution to \eqref{min-unc}, we now discuss further properties of the minimizer. The
following theorem, whose proof is given in Section~\ref{sec-screening} and relies on the maximum principle, deals with
the \textit{screening effect} realized by the optimal configuration: the associated potential vanishes in the uncharged
region. Actually, it turns out that such a property (together with the nonnegativity of the potential) uniquely characterizes the minimizer (see Remark~\ref{rem-screening}).
\begin{theorem}[Screening]\label{thm-screening}
Assume that $\Omega^+\subset\R^3$ is a bounded and open set with Lipschitz boundary.
Let $\Omega^-$ be a representative of the minimizer of \eqref{min-unc} and let $\phi$ be the corresponding potential, defined in \eqref{def-phi}. Then $\phi \geq 0$ in $\R^3$ and
  \begin{align} \label{eq-screen}
    \phi=0 \qquad\text{ almost everywhere in } \R^3\setminus({\Omega^+\cup\Omega^-})\,.
  \end{align}
  After possibly changing $\Ome^-$ on a set of Lebesgue measure zero, we have
  \begin{align} \label{def-Ome-} %
    \Ome^- =  \{  \phi > 0 \} \BS \overline{\Ome^+}\,.
  \end{align}
  If $\Ome^+$ satisfies an interior ball condition, then we also have $\phi>0$ in $\overline{\Omega^+}$.
\end{theorem}
It is convenient to also introduce a notation for the uncharged region: in view of \eqref{def-Ome-}, we set
\begin{align} \label{def-Ome0} %
  \Ome_0 := \R^3\setminus\overline{\{\varphi>0\}}\,,
\end{align}
which by \eqref{def-Ome-} coincides, up to a set of measure zero, with $\R^3\setminus(\overline{\Omega^+\cup\Omega^-})$. Notice that by \eqref{def-Ome-} and \eqref{def-Ome0} we are selecting precise representatives of the sets $\Omega^-$ and $\Omega_0$, which in general are defined up to a set of Lebesgue measure zero, and that with this choice they are open sets.

\medskip

Based on the screening property and classical maximum principles for subharmonic functions, we establish some further qualitative properties on the shape of the minimizer $\Ome^-$.

\begin{theorem}[Structure of $\Ome^-$]\label{thm-support} %
  Suppose that the assumptions of Theorem~\ref{thm-screening} hold and let $\Omega^-$ be the minimizer of problem \eqref{min-unc}, given by \eqref{def-Ome-}. Then $\Ome^-$ is open, bounded and the following statements hold:
  \begin{enumerate}
  \item\label{it-2.4-1} $\dist (x,\Omega^+)\ \leq\ 2\,|\Omega^+|^{1/3}\quad\text{ for all }x\in {\Omega^-}$;
  \item\label{it-2.4-2}  $\diam\Omega^- \ \leq\ (1+2\sqrt{3})\ \diam\Omega^+$;
  \item\label{it-2.4-3}  for every connected component $V$ of ${\Ome^-}$ we have $\partial V \cap\partial\Omega^+\neq\emptyset$;
  \item\label{it-2.4-4}  if $\Ome^+$ satisfies an interior ball condition, then $\dist(\Ome_0,\partial\Omega^+)>0$.
  \end{enumerate}
\end{theorem}

Notice that, as a consequence of Theorem~\ref{thm-screening} and Theorem~\ref{thm-support}, the potential $\phi$ of the minimizing configuration has compact support. 
We complete our analysis of the minimum problem \eqref{min-unc} by discussing some further properties of the minimizer, included the regularity of its boundary. This relies heavily on the observation that, as a consequence of Theorem~\ref{thm-screening}, the potential $\phi$ associated with a minimizer of \eqref{min-unc} is in fact a solution to a classical obstacle problem.
Indeed, as a consequence of the characterization \eqref{def-Ome-}, $\phi$ solves
\begin{align} \label{eq-obstacle1}
  \begin{cases}
    \Delta\phi = \chi_{\{\phi>0\}} & \text{in }\R^3\setminus\overline{\Omega^+},\\
    \phi\geq0\ .
  \end{cases}
\end{align}
It then follows that $\phi$ solves the obstacle problem
\begin{equation*} 
  \min \biggl\{ \int_{\R^3\setminus\overline{\Omega}^+} \Bigl( |\nabla\psi|^2 +2\psi \Bigr)\dds x
  \  : \  \psi\in H^1(\R^3\setminus\overline{\Omega^+}),\   \psi\geq0, \  \psi=\phi \text{ on }\partial\Omega^+ \biggr\}\,,
\end{equation*}
see Proposition~\ref{prp-obstacle}. 

The well-established regularity theory for the so-called \textit{free boundary} of a solution to an obstacle problem also yields more information about the regularity of the boundary of $\Omega_0$ (for a comprehensive account of the available results, see, for instance, the book \cite{PetShaUra}).

\begin{theorem}[Regularity]\label{thm-regularity}
  Under the assumptions of Theorem~\ref{thm-screening}, let $\Omega^-$ be the minimizer of problem \eqref{min-unc}, let $\vphi$ be the associated potential, and let $\Omega_0$ be defined by \eqref{def-Ome0}. Then $\phi\in C^{1,1}_{\mathrm{loc}}(\R^3\setminus\overline{\Omega^+})$ and the boundary of $\Omega_0$ has finite $\mathcal{H}^2$-measure locally in $\R^3\setminus\overline{\Omega^+}$. Moreover, one has the decomposition $\partial\Omega_0 = \Gamma \cup \Sigma$, where $\Gamma$ is relatively open in $\partial\Omega_0$ and real analytic, while $x_0\in\Sigma$ if and only if
  \begin{equation*}
    \lim_{r\to0^+}\frac {\mindiam \bigl( \{\phi=0\}\cap B_r(x_0) \bigr)}{r} =0\,,
  \end{equation*}
  where $\mindiam (E)$ denotes the infimum of the distances between pairs of parallel planes enclosing the set $E$. The Lebesgue density of $\Omega_0$ is 0 at each point of $\Sigma$.
\end{theorem}

The proof of Theorem~\ref{thm-regularity} is given in Section~\ref{sec-obstacle}, and a more precise characterization of the singular points of $\p \Ome_0$ is given in Proposition~\ref{prp-singbound}.
Notice that the only possible singularities allowed in a minimizer are of ``cusp-type'', since the set $\Omega_0$ has zero Lebesgue density at such points. An example of occurrence of a singular point is presented in Remark~\ref{rem-singularity}.

\medskip

In the final section we consider for given $u^+ \in L^1(\R^3;\{0,1\})$ and for $\e>0$ the energy
\begin{align*}
  \mathcal F_\e (u) \ :=\
  \begin{cases}
    \| u^+ - u\|_{H^{-1}(\R^3)}^2 &\text{if } u\in L^1(\R^3;\{0,\frac{1}{\e}\}), \ \int_{\Omega^+} u=0, \ \int_{\R^3} u\leq\lambda,\\
    \infty &\text{ else.}
  \end{cases}
\end{align*}
Here our main result is the Gamma-convergence of $\mathcal F_\e$ to an energy defined on a class of positive Radon
measures, see Theorem~\ref{thm-gammaconv}. Furthermore, in Proposition~\ref{prp-surfacecharge} we show that minimizers
of the limit energy are supported on the boundary $\partial \Omega^+$ and thus describe a surface charge distribution.

\begin{remark}
Although we have restricted our analysis to the physically meaningful case of three dimensions with a Newtonian potential, we believe that the methods used in this paper can be extended in a straightforward way to obtain the corresponding results in higher space dimensions. Indeed, our analysis is based on general tools rather than on the specific three-dimensional structure of the problem. Similarly, it should be possible to treat also more general Riesz kernels $\frac{1}{|x-y|^\alpha}$ in the energy.
\end{remark}

%%%%%%%%%%%%%%%%%%%%%%%%%%%%%%%%%%%%%%%%%%%%%%%%%%%%%%%%%%%%%%%%%%%%%%%%%%%%%%%%%%%%%%%%%%%%%%%%%%%%%%%%%
%%%%%%%%%%%%%%%%%%%%%%%%%%%%%%%%%%%%%%%%%%%%%%%%%%%%%%%%%%%%%%%%%%%%%%%%%%%%%%%%%%%%%%%%%%%%%%%%%%%%%%%%%

\section{Existence and the relaxed problem} \label{sec-relaxed} %

In this section, we give the proofs of Theorems~\ref{thm-unconstrained} and \ref{thm-constrained}.  In order to overcome
the difficulties in the proof of the existence of a minimizer pointed out in the discussion above, it is convenient to
relax the problem by allowing for intermediate densities of charge taking values in $[0,1]$, the convex hull of
$\{0,1\}$.

\medskip

In this section, we will always assume that $\Ome^+$ is an open and bounded set with $|\Omega^+|=m$ for some $m>0$. We also recall that $u^+:=\chi_{\Omega^+}$ is the characteristic function of $\Ome^+$.  We then consider, for $\lambda>0$, the relaxed minimum problem
\begin{align} \label{min-rel} %
  \min \ \biggl\{ \mathcal{E}(u)\ : \ u \in L^1(\R^3;[0,1]), \int_{\Omega^+} u\dds x = 0,\ \int_{\R^3}u\dds x \leq\lambda
  \biggr\}\,,
\end{align}
where
\begin{align*}
  \mathcal{E}(u):=\int_{\R^3}|\nabla\phi|^2\dds x
\end{align*}
and $\phi$ is the potential associated to $u$, defined by \eqref{def-phi}. The corresponding class of admissible configurations is given by
\begin{align*}
  \mathcal{A}_\lambda := \Bigl\{ u \in L^1(\R^3; [0,1]) \ : \ \int_{\Omega^+} u\dds x = 0,\
  \int_{\R^3}u\dds x\leq\lambda \Bigr\}\,.
\end{align*}
We first note that the potential $\phi$ is uniformly bounded and indeed vanishes for $|x| \to \infty$.
This is \textit{a priori} not clear, since $u$ may have unbounded support.

\begin{lemma} \label{lem-decay} %
Assume $u\in\mathcal A_\lambda$. Then the potential $\phi$, defined in \eqref{def-phi}, satisfies 
\begin{align}
	-\frac{1}{2}\Big(\frac{3\lambda}{4\pi}\Big)^{\frac{2}{3}}\,&\leq\,\phi(x)\,\leq\,\frac{1}{2}\Big(\frac{3m}{4\pi}\Big)^{\frac{2}{3}}\quad\text{ for all }x\in\R^3, \label{eq:bounds-phi}\\
	|\phi(x)| \,&\to\, 0\quad\text{ for }|x| \to \infty. \label{eq:decay-phi}
\end{align}
\end{lemma}

\begin{proof}
  For $t>0$ let $r(t)$ denote the radius of a ball with volume $t$, thus $\frac{4\pi}{3}r(t)^3=t$. By classical rearrangement inequalities \cite[Theorem~3.4]{LiLo01} we deduce
  \begin{align*}
  	\varphi(x) \,\leq\, \int_{\R^3} \frac{u^+(y)}{4\pi|x-y|}\dd y \,\leq\, \int_{B_{r(m)}}\frac{1}{4\pi|y|}\dd y \,=\, \frac{r(m)^2}{2}\,.
  \end{align*}
  This shows the upper estimate in \eqref{eq:bounds-phi}. The lower bound follows similarly.
  
  Next let $\e>0$ be given and fix $R_\e>\frac{1}{\e}$ such that $\int_{\R^3\setminus B_{R_\e}}u <\e$ and $\Omega^+\subset
  B_{R_\e}$. Again by rearrangement inequalities we can bound
  \begin{align*}  %
    \int_{\R^3\setminus B_{R_\e}}\frac{u(y)}{4\pi|x-y|}\dds y \leq
    \max_x\max_{\substack{0\leq w \leq 1\\\int w \leq \e}}\int_{\R^3}\frac{w(y)}{4\pi|x-y|}\dds y \leq
    \int_{B_{r(\e)}}\frac{1}{4\pi|y|}\dds y = \frac{r(\e)^2}{2}\,.
  \end{align*}
  Then for every $x$ with $|x|>2 R_\e$ one has
  \begin{align*}
    |\phi(x)| \leq \int_{B_{R_\e}}\frac{|u^+(y)-u(y)|}{4\pi|x-y|}\dds y + \int_{\R^3\setminus
      B_{R_\e}}\frac{u(y)}{4\pi|x-y|}\dds y \leq \frac{m+\lambda}{4\pi R_\e} + \frac{r(\e)^2}{2}\,,
  \end{align*}
  which shows \eqref{eq:decay-phi}.
\end{proof}

Existence and uniqueness of a minimizer for the relaxed problem \eqref{min-rel} follow directly from standard arguments.

\begin{proposition}[Minimizer for the relaxed problem] \label{prp-existence} %
  For every $\lambda > 0$, the relaxed minimum problem \eqref{min-rel} admits a unique solution $u_\lambda \in \mathcal{A}_\lam$.
\end{proposition}
\begin{proof}
  The existence of a minimizer follows by the Direct Method of the Calculus of Variations and standard semicontinuity
  arguments. Indeed, for a minimizing sequence $u_n\in\mathcal{A}_\lambda$ we have that, up to subsequences, $u_n\wto u$
  weakly* in $L^\infty(\R^3)$ for some measurable function $u\in L^\infty(\R^3)$, which is clearly still an element of
  the class $\mathcal{A}_\lambda$.  To prove semicontinuity, we express the total energy as
  \begin{align*}
    \mathcal{E}(u) = \int_{\R^3}\int_{\R^3} \frac{(u^+-u)(x)(u^+-u)(y)}{4\pi|x-y|} \,\dd x\dd y\,.
  \end{align*}
  For the self-interaction energy of $u$ we have
  \begin{align*}
    \int_{\R^3}\int_{\R^3}\frac{u(x)u(y)}{4\pi|x-y|}\dds x\dd y \leq
    \liminf_{n\to\infty}\int_{\R^3}\int_{\R^3}\frac{u_n(x)u_n(y)}{4\pi|x-y|}\dds x\dd y
  \end{align*}
  by classical potential theory (see, for instance, \cite[equation~(1.4.5)]{Lan}).  For the mixed term, we have
  \begin{align*}
    \int_{\R^3}\int_{\R^3}\frac{u^+(x)u_n(y)}{4\pi|x-y|}\dds x\dd y &= \int_{\R^3}\phi^+(y)u_n(y)\dds y \\
    &\to
    \int_{\R^3}\phi^+(y)u(y)\dds y 
    = \int_{\R^3}\int_{\R^3}\frac{u^+(x)u(y)}{4\pi|x-y|}\dds x\dd y
  \end{align*}
  where in passing to the limit we used the fact that the potential $\phi^+$ associated to the positive phase $\Omega^+$
  is a continuous function vanishing at infinity. This completes the proof of existence.
  
  \medskip

  Uniqueness of the minimizer follows by convexity of the problem: let $u_1,u_2\in\mathcal{A}_\lambda$ be two solutions to
  the the relaxed minimum problem \eqref{min-rel}, and let $\phi_1,\phi_2$ be the associated potentials. Setting
  $u_\alpha:=\alpha u_1 + (1-\alpha)u_2$ for $\alpha\in(0,1)$, we have $u_\alpha\in\mathcal{A}_\lambda$ and the associated
  potential is given by $\phi_\alpha = \alpha\phi_1 + (1-\alpha)\phi_2$. Hence
  \begin{align*}
    \mathcal{E}(u_\alpha) = \int_{\R^3}|\nabla\phi_\alpha|^2\dds x < \alpha\int_{\R^3}|\nabla\phi_1|^2\dds x +
    (1-\alpha)\int_{\R^3}|\nabla\phi_2|^2\dds x = \min_{\mathcal{A}_\lambda}\mathcal{E}(u)
  \end{align*}
  unless $\nabla\phi_1=\nabla\phi_2$. Hence $u_1=u_2$ almost everywhere, and the minimizer is unique.
\end{proof}

We now turn our attention to some useful properties of a minimizer of the relaxed problem \eqref{min-rel}, following from first variation arguments.

\begin{lemma}[First variation of the relaxed problem] \label{lem-EL} %
  Assume that $u$ is the minimizer of the relaxed problem \eqref{min-rel} and let $\phi$ be the associated potential.
  Let $\eta$ be any bounded Lebesgue integrable function such that $\int_{\Omega^+} |\eta|=0$.  Then the following
  properties hold:
  \begin{enumerate}
  \item\label{it:lem3.3-1} If $\int_{\R^3}\eta=0$ and there exists $\delta>0$ such that $\supp\eta\subset\{\delta<u<1-\delta\}$, then
    \begin{align*}
      \int_{\R^3}\phi\eta\dds x =0\,.
    \end{align*}
  \item\label{it:lem3.3-2} If $\int_{\R^3}\eta\leq0$ and there exists $\delta>0$ such that $\eta\geq0$ on $\{u<\delta\}$ and $\eta\leq0$ on
    $\{u>1-\delta\}$, then
    \begin{align*}
      \int_{\R^3}\phi\eta\dds x \leq 0\,.
    \end{align*}
  \item\label{it:lem3.3-3} If $\int_{\R^3} u < \lambda$, $\eta\geq 0$, and if there exists $\delta>0$ with
    $\supp\eta\subset \{u<1-\delta\}$ then
    \begin{align*}
      \int_{\R^3}\phi\eta\dds x \leq 0\,.
    \end{align*}
  \end{enumerate}
\end{lemma}

\begin{proof}
  We first prove \ref{it:lem3.3-1}. The function $u_\e:=u\pm\e\eta$, for $\e>0$ sufficiently small, is admissible in the relaxed
  problem \eqref{min-rel}. Let $\psi$ be such that $\Delta\psi=\eta$, so that $\phi_\e=\phi\pm\e\psi$ satisfies
  $-\Delta\phi_\e=u^+ - u_\e.$ Then by minimality of $u$ we have
  \begin{align*}
    \int_{\R^3}|\nabla\phi|^2\dds x \leq \int_{\R^3}|\nabla\phi_\e|^2\dds x \,,
  \end{align*}
  from which, by letting $\e\to0$, we immediately deduce 
  \begin{align*}
    0 = \int_{\R^3}\nabla\phi\cdot\nabla\psi\dds x = -\int_{\R^3} \phi\Delta\psi\dds x = -\int_{\R^3}\phi\eta\dds x\,.
  \end{align*}
  Let now $u,\eta$ satisfy the assumptions in \ref{it:lem3.3-2} or \ref{it:lem3.3-3}.  Then the function $u_\e:=u+\e\eta$, for $\e>0$
  sufficiently small, is admissible in the relaxed problem \eqref{min-rel}, and arguing as before we obtain
  \begin{align*}
    0 \leq \int_{\R^3}\nabla\phi\cdot\nabla\psi\dds x = -\int_{\R^3} \phi\Delta\psi\dds x = -\int_{\R^3}\phi\eta\dds
    x\,,
  \end{align*}
  which completes the proof.
\end{proof}

As a consequence of the first order conditions proved in previous lemma, it follows that the potential associated to a minimizer is everywhere nonnegative.

\begin{lemma}[Nonnegativity of $\vphi$] \label{lem-phipos} %
  Assume that $u$ is the minimizer of the relaxed problem \eqref{min-rel} and let $\phi$ be the associated
  potential. Then $\phi \geq 0$ in $\R^3$.
\end{lemma}

\begin{proof}
  For $\delta > 0$, let $x \in E_\delta := \{ u > \delta \}$ such that $E_\delta$ has positive Lebesgue density at $x$. By an application of Lemma~\ref{lem-EL}\ref{it:lem3.3-2} with $\eta:=-\chi_{E_\delta\cap B_r(x)}$, we then get for every $r>0$
  \begin{align*}
    \int_{E_\delta \cap B_r(x)}\phi(y)\dds y \geq 0\,.
  \end{align*}
  Since $\phi$ is continuous, it follows that $\phi(x)\geq 0$ for all $x \in E_\delta := \{ u > \delta \}$ such that
  $E_\delta$ has positive Lebesgue density at $x$. By \cite[Corollary~2.14]{Mat95}, we hence have $\phi \geq 0$ a.e. in
  $E_\delta$. Since $\delta > 0$ is arbitrary, it follows that $\phi \geq 0$ a.e. in $E_0 := \{ u > 0 \}$. By changing
  $u$ on a set of Lebesgue measure zero, we hence may assume that $\phi \geq 0$ in $E_0$.

  \medskip

  By the above calculation, the open set $U:=\{\phi<0\}$ is contained in $\{ u \leq 0 \}$, and hence $-\Delta\phi\geq0$
  in $U$.  Since $\phi$ vanishes at the boundary of $U$ and at infinity, by the minimum principle we conclude that
  $\phi$ must be nonnegative in $U$, which is a contradiction unless $U=\emptyset$. This shows that $\phi\geq0$ in
  $\R^3$.
\end{proof}

The following simple lemma is used in the proofs of Proposition~\ref{prp-saturation} and Theorem~\ref{thm-screening}.

\begin{lemma} \label{lem-potential} %
  Let $w\in L^1(\R^3)\cap L^\infty(\R^3)$ and let $\phi$ be the associated potential, that is
  \begin{align*}
    \phi(x) := \int_{\R^3} \frac{w(y)}{4\pi|x-y|}\dds y\,.
  \end{align*}
  Then
  \begin{align}\label{eq-pot1}
    \int_{\partial B_R}\phi\dds\mathcal{H}^2 = \int_R^{\infty} \frac{R^2}{r^2} \biggl( \int_{B_r}w\dds x \biggr)\dds r\,.
  \end{align}
  In particular, if $\supp w\subset B_R$ for some $R>0$ and $\int_{\R^3}w=\lambda$, then
  \begin{align} \label{eq-pot2} %
    \int_{\partial B_R}\phi \dds\mathcal{H}^2 = \lambda R\,.
  \end{align}
\end{lemma}

\begin{proof}
  Since $-\Delta\phi=w$, we have
  \begin{align} \label{prev-eq} %
    \frac{\dd}{\dd R}\med_{\partial B_R} \phi\dds\mathcal{H}^2 = \med_{\partial B_R}\frac{\partial\phi}{\partial\nu}\dds\mathcal{H}^2 = \frac{1}{4\pi R^2}
    \int_{B_R}\Delta\phi\dds x = - \frac{1}{4\pi R^2}\int_{B_R}w\dds x\,.
  \end{align}
  Integrating \eqref{prev-eq} between $R$ and $\infty$ and recalling that $\lim_{R\to\infty}\med_{\partial B_R}\phi=0$ by Lemma~\ref{lem-decay}, we obtain the conclusion.
\end{proof}

We next use the first variation formulas to show that minimizers of the relaxed problem \eqref{min-rel} minimize the absolute value of the total net charge within the set of admissible configurations.

\begin{proposition}[Saturation of charges] \label{prp-saturation}
For every $\lambda>0$, the solution $u_\lambda$ of the relaxed minimum problem \eqref{min-rel} with $|\Ome^+| = m$ satisfies
\begin{align} \label{sat-1} %
\int_{\R^3} u_\lambda \dds x \ = \ \min \{ \lambda, m \}\,.
\end{align}
Furthermore, for all $\lam \geq m$, we have $u_\lam = u_m$.
\end{proposition}

\begin{proof}
  Denote by $\phi$ the potential of $u^+-u_\lambda$ as in \eqref{def-phi} and choose $R_0>0$ such that $\Omega^+\subset B_{R_0}$. Arguing by contradiction, we first assume
  \begin{align} \label{ass-fals} %
    \mu \ :=\ \int_{\R^3} u_\lam \dds x \ <\ \min \{ \lambda, m \}\,.
  \end{align}
  Our argument is based on the fact that screening is not possible under the assumption \eqref{ass-fals}. Indeed, we will even show
  \begin{align} \label{no-screening} %
    \big| B_R^c\cap \{\phi>0\}\cap \{u_\lam<1\} \big| \ =\ \infty \quad\text{ for all }R>0.
  \end{align}
  We first note that \eqref{ass-fals} yields a contradiction if \eqref{no-screening} holds. Indeed, by \eqref{no-screening} we can choose $\delta>0$, $R>R_0$ such that
  $\big| (B_{R+1}\setminus B_R) \cap \{\phi>0\} \cap \{u_\lam<1-\delta\} \big| >0$.
  Letting $\eta:= \Chi_{(B_{R+1}\setminus B_R)\cap \{u_\lam<1-\delta\}}$, by Lemma~\ref{lem-EL}\ref{it:lem3.3-3} we deduce that
  \begin{align*}
    0\ \geq\ \int_{\R^3} \phi\eta \,\dd x \ =\ \int_{(B_{R+1}\setminus B_R)\cap \{u_\lam<1-\delta\}} \phi\,\dd x \ >\
    0\,,
  \end{align*}
  which is impossible. This shows that $\int_{\R^3} u_\lam \geq \min \{ \lambda, m \}$ and proves \eqref{sat-1} for $\lam \leq m$, since $u_\lambda\in \mathcal A_\lambda$.
  
  We next give the argument for \eqref{no-screening}. By \eqref{eq-pot1}, we have for all $R>R_0$
  \begin{align*}
    \int_{\partial B_R}\phi\dds\Ha^2
    = \int_R^\infty \frac{R^2}{r^2} \Big(\int_{B_r} (u^+-u_\lam) \dds x \Big)\dds r
    \geq\ \int_R^\infty \frac{(m-\mu)R^2}{r^2}\dds r = (m-\mu)R\,.
  \end{align*}
  By integrating this identity from $R$ to $\infty$, we get $\int_{B_R^c} \phi \ =\ \infty$. Since $\phi$ is uniformly bounded, this implies $|B_R^c\cap\{\phi>0\}|=\infty$. On the other hand, we have $|\{u_\lam=1\}|\ \leq\ \int_{\R^3}u_\lam \ =\ \mu$, which yields \eqref{no-screening}. %and completes the proof in the case $\lambda\leq m$.

  \medskip

  It remains to consider the case $\lam > m$ and to show that the assumption
  \begin{align} \label{ass-fals2}
    \mu \ :=\ \int_{\R^3} u_\lambda \dds x > m
  \end{align}
  yields a contradiction. 
  By Lemma~\ref{lem-phipos}, we have $\phi\geq 0$ in $\R^3$. On the other hand, by the proof of Lemma~\ref{lem-potential} we get
  \begin{align*}
    \frac{\dd}{\dd R} \med_{\partial B_R} \phi \dds\mathcal{H}^2 %
    &\lupref{prev-eq}= -\frac{1}{4\pi R^2} \int_{B_R} (u^+-u_\lambda)\dds x = \frac{1}{4\pi R^2} \biggl(\mu - m -\int_{\R^3\setminus B_R}u_\lambda\dds x \biggr)\,.
  \end{align*}
  Since by \eqref{ass-fals2} the last term is positive for $R$ sufficiently large, it follows that the mean value $\med_{\partial B_R}\phi$ of $\phi$ on spheres is a strictly increasing function of the radius, for large radii, vanishing in the limit as $R\to\infty$. This is clearly in contradiction with the fact that $\phi\geq0$. We conclude that $\int u_\lambda = m$ for $\lambda\geq m$ and, in turn, $u_\lambda=u_m$ by uniqueness of the minimizer.
\end{proof}

The previous proposition allows us to draw some conclusions on the dependence of the minimal energy in the relaxed problem \eqref{min-rel} on the parameter $\lambda$.

\begin{corollary}[Minimal energy as function of $\lam$] \label{cor-e} %
  For $\lambda>0$ let $e(\lambda):=\mathcal{E}(u_\lambda)=\min_{\mathcal{A}_\lambda}\mathcal{E}$. Then $e(\lambda)$ is continuous, strictly decreasing for $\lambda\in[0,m]$ and constant for $\lambda\geq m$.
\end{corollary}

\begin{proof}
Since $\mathcal A_\lambda\subset\mathcal A_{\lambda'}$ for $\lambda\leq \lambda'$ the minimal energy $e(\lambda)$ is decreasing. The strict monotonicity of $e(\lambda)$, for $\lambda\leq m$, follows from the fact that if $e(\lambda)=e(\lambda')$ for some $0<\lambda<\lambda'\leq m$, then the uniqueness of minimizers would imply that $u_\lambda=u_{\lambda'}$, which is not permitted by \eqref{sat-1}. The fact that $e(\lambda)$ is constant for $\lambda\geq m$ follows also from Proposition~\ref{prp-saturation}.
  
  To prove that $e$ is continuous, we observe that for $\lambda'>\lambda$ we can use the function
  $\frac{\lambda}{\lambda'}u_{\lambda'} \in \mathcal{A}_\lambda$ as a competitor in the relaxed minimum problem \eqref{min-rel}, which yields $e(\lambda)\leq\mathcal{E}(\frac{\lambda}{\lambda'}u_{\lambda'})$. Hence, by monotonicity
  \begin{align*}
    e(\lambda)\leq\liminf_{\lambda'\searrow\lambda} \mathcal{E}\Bigl(\frac{\lambda}{\lambda'}u_{\lambda'}\Bigr) =
    \liminf_{\lambda'\searrow\lambda} e(\lambda') \leq e(\lambda)\,,
  \end{align*}
  which implies continuity from the right. Similarly, by considering $\lambda''<\lambda$ and comparing with
  $\frac{\lambda''}{\lambda}u_\lambda \in \mathcal{A}_{\lambda''}$ we obtain continuity from the left. Together this shows that $e$ is continuous.
\end{proof}

We next address the proofs of Theorem~\ref{thm-constrained} and Theorem~\ref{thm-unconstrained} on existence and uniqueness for the constrained and unconstrained minimum problems.

\begin{proof}[Proof of Theorem~\ref{thm-constrained}]
  We divide the proof into two steps.
  
  \smallskip
  
  {\it Step 1: the case $\lam \leq m$}. By Proposition~\ref{prp-existence} there exists a unique minimizer $u$ of $\mathcal E$ in the class of densities $\mathcal{A}_\lambda$. By Proposition~\ref{prp-saturation}, we have $\int_{\R^3}u\ =\ \lam$.  It therefore remains to show that the set $\{0<u<1\}$ has zero Lebesgue-measure. Arguing by contradiction, we assume that $|\{0<u<1\}|>0$. Then there exists $\delta>0$ such that the set
  $\mathcal{U}_\delta:=\{\delta<u<1-\delta\}$ has positive measure. We set $\eta(x):=(\phi(x)-c)\chi_{\mathcal{U}_\delta}(x)$,
  where
  \begin{align} \label{avg-con} %
    c:=\med_{\mathcal{U}_\delta} \phi(x)\dds x\,.
  \end{align}
  The function $\eta$ satisfies the assumptions of Lemma~\ref{lem-EL}\ref{it:lem3.3-1}, and we deduce
  \begin{align*}
    0=\int_{\R^3}\phi\eta\dds x = \int_{\mathcal{U}_\delta}\phi(\phi-c)\dds x \lupref{avg-con}=
    \int_{\mathcal{U}_\delta}(\phi-c)^2\dds x\,,
  \end{align*}
  which implies $\phi=c$ almost everywhere in $\mathcal{U}_\delta$. 
  By \eqref{def-phi} and standard elliptic theory, we have $\phi \in W^{2,p}_{\mathrm{loc}}(\R^3)$ for all $p < \infty$. From Stampacchia's Lemma \cite[Proposition~3.23]{GiMa12}, one can deduce that $\nabla\phi=0$ almost everywhere in $\mathcal{U}_\delta$ and then that $\Delta \phi=0$ almost everywhere in $\mathcal{U}_\delta$. On the other hand $\Delta\phi =u$ in $\mathcal{U}_\delta$, which contradicts our assumption.

  \medskip

  {\it Step 2: the case $\lam > m$}. Suppose that there exists a minimizer $\Omega^-$ of \eqref{min-con}, and let $u:=\chi_{\Omega^-}$. By Proposition~\ref{prp-saturation} the unique minimizer $u_\lambda$ of the corresponding relaxed problem \eqref{min-rel} is given by $u_\lambda=u_m$, in particular we have
  $$
  \EE(u_m) < \EE(u)\,.  
  $$
  Moreover, by the previous step, $u_m$ is in fact the characteristic function of a set. 
  For $R > 0$ let $u^R := u_m + (1-u_m)\chi_{B_R\setminus B_{\t R}}$, where $\t R= \t R(R) > 0$ is chosen such that $\int_{\R^3} u^R=\lambda$, which is equivalent to the condition
\begin{align*}
	\lambda \,=\, m + \frac{4\pi}{3}(R^3-\tilde R^3) - \int_{B_R\setminus B_{\t R}} u_m\ \dd x\,.
\end{align*}
We deduce first that $\t R(R)\to\infty$ as $R\to\infty$ and then $\frac{4\pi}{3}(R^3-\tilde R^3)=\lambda-m +o(1)$ as $R\to\infty$. Since $\Omega^+$ is bounded, and since $u_m$ takes values in $\{0,1\}$ almost everywhere we deduce that for $R$ sufficiently large $u^R$ is the characteristic function of an admissible set for the minimizing problem \eqref{min-con}. We claim that $\EE(u^R) \to \EE(u_m) < \EE(u)$. For $R$ large enough, this yields a contradiction to the statement that $u$ is a minimizer of \eqref{min-con}. To prove the convergence of $\EE(u^R)$ observe that an explicit calculation for the self-interaction energy of $\chi_{B_R\setminus B_{\t R}}$ yields
\begin{align*}
	\int_{B_R\setminus B_{\t R}}\int_{B_R\setminus B_{\t R}}\frac{1}{4\pi|x-y|} \dd y\dd x \,&=\,
	c(3\t R^5+2R^5-5\t R^3R^2) \\
	&=\, 15cR^3\delta^2 + c_1 R^2\delta^3 + c_2R\delta^4 + c_3\delta^5
\end{align*}
(see \eqref{eq:ring} in the Appendix), with $\delta=R-\t R=\frac{\lambda-m}{4\pi R^2} +\mathcal{O}(R^{-3})$.
This implies that the self-interaction energy of the annulus vanishes as $R\to\infty$. By a similar asymptotic analysis one shows that the interaction energy of the annulus with the charge distributions $u^+,u_m$ also tends to zero as $R\to\infty$.
\end{proof}

\begin{proof}[Proof of Theorem~\ref{thm-unconstrained}] %
By Theorem~\ref{thm-constrained}, there exists a unique minimizer $\Omega^-$ of the constrained minimum problem \eqref{min-con} for $\lambda=m$. We claim that $\Omega^-$ is also the unique solution of the unconstrained problem \eqref{min-unc}. Indeed, the conclusion follows immediately from the fact that $\chi_{\Omega^-}$ is the unique solution of the relaxed problem \eqref{min-rel} for $\lambda= m$, and by Proposition~\ref{prp-saturation} and Corollary \ref{cor-e}.
\end{proof}

%%%%%%%%%%%%%%%%%%%%%%%%%%%%%%%%%%%%%%%%%%%%%%%%%%%%%%%%%%%%%%%%%%%%%%%%%%%%%%%%%%%%%%%%%%%%%%%%%%%%%%%%%
%%%%%%%%%%%%%%%%%%%%%%%%%%%%%%%%%%%%%%%%%%%%%%%%%%%%%%%%%%%%%%%%%%%%%%%%%%%%%%%%%%%%%%%%%%%%%%%%%%%%%%%%%

\section{Proof of the screening property} \label{sec-screening} %

We now turn to the proof of Theorem~\ref{thm-screening} related to the screening of the positive charge.

\begin{proof}[Proof of Theorem~\ref{thm-screening}]
The nonnegativity of $\vphi$ is proved in Lemma~\ref{lem-phipos}.
We introduce the closed sets
\begin{align*}
	A^- \,&=\, \{x\in\R^3\,:\, |B_r(x)\cap\Omega^-|>0\,\text{ for all }r>0\}\,,\\
	A_0 \,&=\, \{x\in\R^3\,:\, |B_r(x)\setminus(\Omega^+\cup\Omega^-)|>0\,\text{ for all }r>0\}\,.
\end{align*}
Notice that $A^-$ and $A_0$ are independent of the precise representative $\Omega^-$ of the minimizer: indeed, they coincide with the closures of the sets of points of positive Lebesgue density of $\Omega^-$ and of $\R^3\setminus(\Omega^+\cup\Omega^-)$, respectively. Furthermore, $|\Omega^-\setminus A^-|=|(\R^3\setminus(\Omega^+\cup\Omega^-))\setminus A_0|=0$. We have selected the sets $A^-$ and $A_0$ as the largest possible sets for which we can apply the comparison argument in Step ~1 below. We now divide the proof of the theorem into four steps.

\smallskip

{\it Step 1.} In this step of the proof, we show that
\begin{align} \label{screen-step1} %
\sup_{A_0}\phi \leq \inf_{A^-}\phi\,.
\end{align}
Indeed, for \eqref{screen-step1} it is sufficient to prove that $\phi(x_0)\leq\phi(x_1)$ for every pair of points $x_0\in A_0$ and $x_1\in A^-$. Define a variation field $\eta\in L^1(\R^3;[0,1])$ by
\begin{align*}
\eta(x):=
\begin{cases}
\frac{1}{|B_r(x_0)\setminus(\Omega^+\cup\Omega^-)|} &\text{if }x\in B_r(x_0)\setminus(\Omega^+\cup\Omega^-), \\
-\frac{1}{|B_r(x_1)\cap\Omega^-|} &\text{if }x\in B_r(x_1)\cap\Omega^-, \\
0 & \text{otherwise}.
\end{cases}
\end{align*}
Then $\int_{\R^3} \eta = 0$, and, by an application of Lemma~\ref{lem-EL}(ii), we hence get
\begin{align*}
\med_{B_r(x_0)\setminus(\Omega^+\cup\Omega^-)}\phi(y)\dds y \leq \med_{B_r(x_1)\cap\Omega^-} \phi(x)\dds x\,.
\end{align*}
Since $\phi$ is continuous we deduce by letting $r\downarrow 0$ that $\phi(x_0)\leq \phi(x_1)$, which completes the proof of \eqref{screen-step1}.
  
\medskip
  
{\it Step 2.} We next show that
\begin{align} \label{screen-step2} %
\inf_{A^-}\phi =0\,.
\end{align}
Since $\varphi\geq 0$ this implies by \eqref{screen-step1} that $\vphi=0$ in $A_0$, which proves \eqref{eq-screen}.
  
To prove \eqref{screen-step2} we first consider the case in which $A^-$ is unbounded. Then there is a sequence $(x_n)_n$ in $A^-$ with $|x_n|\to\infty$. In view of Lemma~\ref{lem-decay}, this implies $\phi(x_n)\to 0$ and hence \eqref{screen-step2}.
  
It remains to consider the case when $A^-$ is bounded: let $R_0>0$ be such that $\Omega^+\cup A^-\subset B_{R_0}$.  Then \eqref{eq-pot2} yields $\int_{\partial B_R}\phi=0$ for every $R\geq R_0$, and since $\phi$ is nonnegative we obtain that $\phi\equiv0$ in $\R^3\setminus B_{R_0}$. Next we define the ``interior'' (in a measure-theoretic sense) of $A_0$, that is the open set
\begin{align*}
  	\tilde A_0 \,:=\, \{x\in\R^3\,:\, |B_r(x)\cap (\Omega^-\cup\Omega^+)|=0\,\text{ for some }r>0\}\subset A_0\,.
\end{align*}
As $\phi$ is harmonic in $\t A_0$ and $\phi\equiv0$ in $\R^3\setminus B_{R_0}$, we deduce that $\phi$ vanishes in the closure of the connected component $D$ of $\t A_0$ that contains $\R^3\setminus B_{R_0}$. If $\partial D\cap A^-\neq\emptyset$, we immediately obtain \eqref{screen-step2}.

It therefore remains to consider the case $\partial D\cap A^-=\emptyset$, and we now show that this case in fact never happens. Hence we argue by contradiction assuming that $\partial D\cap A^-=\emptyset$.
  
We first deduce that $\partial D \subset\partial\Omega^+$. In fact, by the assumption $\partial D\cap A^-=\emptyset$ any $x\in\partial D\setminus \partial\Omega^+$ has positive distance to the sets $\overline{\Omega^+}$ and $A^-$, which shows that $x\in \t A_0$ and implies that $x\not\in\partial D$ as $\t A_0$ is open. Therefore $\partial D\setminus \partial\Omega^+=\emptyset$.

We claim next that there exist a point $x_0\in\partial\Omega^+$ with $\vphi(x_0)=0$ and a ball $B_R\subset\Omega^+$ with $\partial B_R\cap\partial\Omega^+=\{x_0\}$ (we can assume without loss of generality that the ball is centered at the origin). In fact (note that at this stage we do not assume an inner sphere condition), choose any $x_*\in \partial D$. Without loss of generality we can assume that
\begin{align*}
  \Omega^+\cap B_r(x_*) = \{(y,t)\in \R^{2}\times\R\,:\, t>\psi(y)\}\cap B_r(x_*)
\end{align*}
for some Lipschitz function $\psi:\R^{2}\to\R$. Since $x_*\not\in A^-$ by the contradiction assumption, after possibly decreasing $r$ we obtain that $\varphi=0$ on $\graph(\psi)\cap B_r(x_*)$. Choose now any $x_1\in \Omega^+\cap B_r(x_*)$ with $R:=\dist(x_1,\graph(\psi))< \dist(x_1,\partial B_r(x_*))$. Then there exists $x_0\in \graph(\psi)\cap \partial B_R(x_1)$ and we deduce that $x_0$ and the ball $B_R(x_1)$ enjoy the desired properties (see Figure~\ref{fig-screen}).

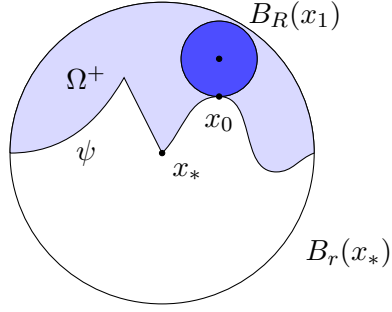
\begin{figure}
\begin{tikzpicture}[scale=0.5]
\draw (0,0) circle [radius=4];
\draw [fill=blue!15!white] (-4,0) to [out=0, in=240] (-1,2) to (0,0) to [out=45, in=180] (1.5,1.5) to [out=0, in=180] (3,-0.5) to [out=0, in=200] (4,0) to [out=90, in=0] (0,4) to [out=180, in=90] (-4,0);
\draw [ultra thin, fill=white] (1.5,2.5) circle [radius=1];
\draw [fill=blue!70!white] (1.5,2.5) circle [radius=1];
\draw [fill] (0,0) circle [radius=0.07];
\node [below right] at (0,0) {$x_*$};
\draw [fill] (1.5,2.5) circle [radius=0.07];
\draw [fill] (1.5,1.5) circle [radius=0.07];
\node [below] at (1.5,1.3) {$x_0$};
\node [below right] at (3.5,-2) {$B_r(x_*)$};
\node [above right] at (2,3) {$B_R(x_1)$};
\node at (-2,2) {$\Omega^+$};
\node at (-2,0) {$\psi$};
\end{tikzpicture}
\caption{The construction of an interior ball touching $\partial\Omega^+$ at a point $x_0$, used in the second step of the proof of Theorem~\ref{thm-screening}.}
\label{fig-screen}
\end{figure}

Since $-\Delta\vphi=1$ in $B_R$ and since $\phi$ is not constant in $B_R$ by Stampacchia's Lemma \cite[Proposition~3.23]{GiMa12}, the minimum principle shows that $\vphi>0$ in $B_R$. Then the Hopf boundary point Lemma \cite[Lemma~3.4]{GiTr01} further implies that $\partial_\nu\vphi(x_0)<0$ for $\nu=\frac{x_0}{|x_0|}$. Since $\vphi$ is of class $C^1$ we conclude that $\vphi(x_0+t\nu)<0$ for $t>0$ sufficiently small, which is a contradiction and completes the proof of claim \eqref{screen-step2} and, in turn, of \eqref{eq-screen}.
  
\medskip

{\it Step 3.} We now prove \eqref{def-Ome-} by showing that
\begin{align} \label{screen-step3} %
\big| \bigl(\Omega^- \bigtriangleup \{\phi>0\} \bigr) \setminus \Omega^+ \big|=0\,.
\end{align}
By \eqref{eq-screen}, we have $\phi(x) = 0$ for almost every $x \not\in {\Ome^+ \cup \Ome^-}$ and hence $\{ \phi > 0 \} \subset \Omega^+ \cup \Omega^-$ up to a Lebesgue nullset.
It remains to show that the set $U = \{\phi=0\} \cap \Ome^-$ satisfies $|U|=0$. Indeed, recalling that $\phi \in W^{2,p}_{\mathrm{loc}}(\R^3)$ for all $p < \infty$, using Stampacchia's Lemma \cite[Proposition~3.23]{GiMa12} as in the proof of Theorem~\ref{thm-constrained} we obtain $\nabla\phi=0$ almost everywhere in $U$ and then that $\Delta \phi=0$ almost everywhere in $U$. Since on the other hand $\Delta\phi = 1$ in $U$, this implies $|U| = 0$. The above arguments together yield \eqref{screen-step3}.

\medskip

{\it Step 4.} We finally show that
\begin{align} \label{screen-step4} %
\min_{\overline{\Omega}^+}\phi>0\,.
\end{align}
under the assumption that $\Ome^+$ satisfies the interior ball condition.
Indeed, if \eqref{screen-step4} does not hold, then we have $\min_{\overline{\Omega}^+}\phi = 0$. By the minimum principle and since $-\Delta\phi=1$ in $\Omega^+$, there is $x_0\in\partial\Omega^+$ such that $\phi(x_0)=0$. By the interior ball condition there exists $B_R(x_1)\subset\Omega^+$ with $\partial B_R(x_1)\cap\partial\Omega^+=\{x_0\}$. But then we can argue as in Step 2 above and obtain a contradiction to the fact that $\phi\geq0$.  This proves \eqref{screen-step4}.
\end{proof}

\begin{remark}
A-posteriori we can identify the set $\tilde A_0$ used in the previous proof with the set $\Omega_0$ defined in \eqref{def-Ome0}. To show this we fix the representative \eqref{def-Ome-} for $\Omega^-$. Since $\Omega^+$ has Lipschitz boundary $|\partial\Omega^+|=0$ holds and hence $\Omega^-\cup\Omega^+=\{\phi>0\}\cup\Omega^+$ up to a set of measure zero. By Stampacchia's Lemma \cite[Proposition~3.23]{GiMa12} we deduce as above that $|\Omega^+\cap\{\phi=0\}|=0$ and thus $\Omega^-\cup\Omega^+=\{\phi>0\}$ up to a set of measure zero. This proves that
\begin{align*}
	\tilde A_0 \,=\, \{x\in\R^3\,:\, |B_r(x)\cap\{\phi>0\}|=0\text{ for some }r>0\}
\end{align*}
and since $\{\phi>0\}$ is open
\begin{align*}
	\tilde A_0 \,=\, \{x\in\R^3\,:\, B_r(x)\cap\{\phi>0\}=\emptyset \text{ for some }r>0\}\,=\,\Omega_0.
\end{align*}
\end{remark}

\begin{remark} \label{rem-screening}
The screening property uniquely characterizes the minimizer, in the following sense: there exists a unique set $\Omega^-$ (up to a set of Lebesgue measure zero) such that the corresponding potential is nonnegative and vanishes outside $\Omega^+\cup\Omega^-$. Indeed, assume by contradiction that there exist two sets $\Omega^-_1$, $\Omega^-_2$ disjoint from $\Omega^+$. We set $u_1=\chi_{\Omega^-_1}$, $u_2=\chi_{\Omega^-_2}$ and let $\phi_1$ and $\phi_2$ be the corresponding potentials characterized by
\begin{align*}
\begin{cases}
-\Delta\phi_i = \chi_{\Omega^+} - \chi_{\Omega^-_i},\\
\lim_{|x|\to\infty}\phi_i(x)=0.
\end{cases}
\end{align*}
Then one has $\phi_i\geq0$ and $\vphi_i=0$ almost everywhere in $\R^3\setminus(\Omega^+\cup\Omega^-_i)$, $i=1,2$.  Hence, $-\Delta (\phi_1-\phi_2)= -(u_1-u_2)$ and testing this equation by $\phi_1-\phi_2$ gives
\begin{align*}
	\int_{\R^3} |\nabla(\phi_1-\phi_2)|^2\dds x \,=\, -\int_{\R^3} (u_1-u_2)(\phi_1-\phi_2)\dds x \,\leq\, 0\,,
\end{align*}
where in the last step we have used the screening property. Therefore $\phi_1-\phi_2$ is constant and since both vanish at infinity we deduce that $\phi_1=\phi_2$, which implies $|\Omega^-_1\Delta\Omega^-_2|=0$.
\end{remark}
Using the screening property we now further characterize $\Omega^-$ and in particular show that this set is essentially bounded.

\begin{proof}[Proof of Theorem~\ref{thm-support}]
The fact that $\Omega^-$, defined by \eqref{def-Ome-}, is open follows by continuity of $\vphi$.
We now turn to the proofs of the other statements.

\medskip

{\it Proof of \ref{it-2.4-1}.}  
We first recall that by the positivity of $\vphi$ and \eqref{eq:bounds-phi}  
\begin{align*} %\label{eq-phi+bound2}
	0\,\leq\, \phi(x) \,\leq\, \frac{3^{2/3}}{2(4\pi)^{2/3}}m^{\frac{2}{3}}\quad\text{ for all }x\in\R^3
\end{align*}
holds. We now adapt the proof of \cite[Lemma~1]{Wei}. Consider $x_0\in {\{\phi>0\}}\setminus\overline{\Omega^+}$ and observe that for every $r<\dist(x_0,\Omega^+)$ we have $\overline{B_{r}(x_0)}\subset \R^3\setminus\overline{\Omega^+}$.
By the screening property, the function $w(x):=\vphi(x)-\frac16|x-x_0|^2$ is harmonic in $B_r(x_0)\cap\{\vphi>0\}$, and the maximum principle yields
$$
\max_{\partial(B_r(x_0)\cap\{\vphi>0\})} w \geq w(x_0) = \vphi(x_0)  >0\,.
$$
Since $w(x)\leq 0$ on $\partial\{\vphi>0\}$, we obtain
\begin{align*}
	0\,\leq\,\varphi(x_0) \,\leq\, \max_{\partial B_r(x_0)} w \,\leq\, \frac{3^{2/3}}{2(4\pi)^{2/3}}m^{\frac{2}{3}} -\frac{r^2}{6}\,,
\end{align*}
thus
\begin{align}
	r^2  \,\leq\, \frac{3\cdot 3^{2/3}}{(4\pi)^{2/3}}m^{\frac{2}{3}} \, \leq\, \frac{3}{2}m^{\frac{2}{3}}\,. \label{eq:boundphi}
\end{align} 
Letting $r\nearrow \dist(x_0,\Omega^+)$ we obtain \ref{it-2.4-1}.

\medskip

{\it Proof of \ref{it-2.4-2}.} By using $m\leq\frac43\pi(\frac12\diam\Omega^+)^3$ in \eqref{eq:boundphi} we easily obtain the estimate in \ref{it-2.4-2}.

\medskip

{\it Proof of \ref{it-2.4-3}.} Let $V$ be a connected component of $\Omega^-$, and assume by contradiction that $\partial V\cap\partial\Omega^+=\emptyset$. Then $-\Delta\phi\leq 0$ in $V$ and $\phi=0$ on $\partial V$, which implies by the maximum principle that $\phi\leq0$ in $V$, which is a contradiction. 

\medskip

{\it Proof of \ref{it-2.4-4}.} The fact that $\partial\Omega_0$ and $\partial\Omega^+$ have positive distance is a consequence of the continuity of $\phi$ and Theorem~\ref{thm-screening}.
\end{proof}

%%%%%%%%%%%%%%%%%%%%%%%%%%%%%%%%%%%%%%%%%%%%%%%%%%%%%%%%%%%%%%%%%%%%%%%%%%%%%%%%%%%%%%%%%%%%%%%%%%%%%%%%%
%%%%%%%%%%%%%%%%%%%%%%%%%%%%%%%%%%%%%%%%%%%%%%%%%%%%%%%%%%%%%%%%%%%%%%%%%%%%%%%%%%%%%%%%%%%%%%%%%%%%%%%%%

\section{Formulation as obstacle problem and regularity of minimizers} \label{sec-obstacle}

In the following proposition we show that, as a consequence of \eqref{eq-obstacle1}, the potential $\phi$ associated with a minimizer of \eqref{min-unc} can be characterized as the solution of an obstacle problem.

\begin{proposition}[Formulation as obstacle problem] \label{prp-obstacle} %
Let $\Omega^+$ be as in Theorem~\ref{thm-screening}, and let $D:=B_{R_0}\setminus\overline{\Omega^+}$, where $R_0$ is chosen so that $\overline{\Omega^+}\subset B_{R_0}$. Then the potential $\phi$ associated with the minimizer $\Omega^-$ of \eqref{min-unc} is the unique solution to the obstacle problem
\begin{align} \label{eq-obstacle2} \min \biggl\{ \int_{D} \Bigl( |\nabla\psi|^2 +2\psi \Bigr)\dds x \ : \ \psi\in H^1(D),\ \psi\geq0, \ \psi-\phi \in H^1_0(D) \biggr\} \,.
\end{align}
\end{proposition}

\begin{proof}
Existence and uniqueness of a solution of \eqref{eq-obstacle2} can be easily established by the direct method of the Calculus of Variations, and by strict convexity of the functional. Moreover, one can show (see, for instance, \cite[Section~1.3.2]{PetShaUra} for details) that the solution $\psi$ belongs to $W^{2,p}_{\mathrm{loc}}(D)$ for every $1<p<\infty$, and that it solves the Euler-Lagrange equations
  \begin{align} \label{eq-obstacle3}
    \begin{cases}
      \Delta\psi = \chi_{\{\psi>0\}} & \text{in }D,\\
      \psi\geq0 & \text{in }D,\\
      \psi=\phi & \text{on }\partial D.
    \end{cases}
  \end{align}
The conditions in \eqref{eq-obstacle3} completely characterize the minimizer of \eqref{eq-obstacle2}: indeed, if $\psi_1,\psi_2\in H^1(D)$ were two different solutions of \eqref{eq-obstacle3}, then for every $\eta\in H^1_0(D)$ we would have
  \begin{align*}
    \int_D \Bigl( \nabla\psi_i\cdot\nabla\eta + \eta\chi_{\{\psi_i>0\}} \Bigr) \dds x=0\,, \qquad i=1,2.
  \end{align*}
Using $\eta=\psi_1-\psi_2$ as a test function and subtracting the two resulting equations, we would get
  \begin{align*}
    0 = \int_D \Bigl( |\nabla(\psi_1-\psi_2)|^2 + (\psi_1-\psi_2) \bigl( \chi_{\{\psi_1>0\}} - \chi_{\{\psi_2>0\}}
    \bigr) \Bigr) \dds x \geq \int_D |\nabla(\psi_1-\psi_2)|^2\dds x\,,
  \end{align*}
which implies $\psi_1=\psi_2$.  Hence, since $\phi$ itself is a solution to \eqref{eq-obstacle3} by \eqref{eq-obstacle1} (which, in turn, follows from the screening property), we conclude that $\vphi$ is also the minimizer of \eqref{eq-obstacle2}.
\end{proof}

We next exploit the connection to the obstacle problem to deduce regularity properties of the free boundary.

\begin{proof}[Proof of Theorem~\ref{thm-regularity}]
By Proposition~\ref{prp-obstacle} the potential $\phi$ is the solution to the obstacle problem \eqref{eq-obstacle2} and solves equation \eqref{eq-obstacle1}. This problem has been widely investigated and we collect below the main results available in the literature, for whose proofs we refer the reader to the presentation in the book \cite{PetShaUra} and to the references contained therein (see, in particular, \cite{Caf77, Caf80, Caf98, Wei}).

First of all, by \cite[Theorem~2.3]{PetShaUra} one has that $\phi\in C^{1,1}_{\mathrm{loc}}(\R^3\setminus\overline{\Omega^+})$.
The \textit{free boundary} $\Gamma(\phi):=\partial\{\phi>0\}$ has locally finite $\mathcal{H}^2$-measure in $\R^3\setminus\overline{\Omega^+}$ by \cite[Lemma~3.13]{PetShaUra}.

Moreover, by \cite[Theorem~3.22, Theorem~3.23 and Definition~3.24]{PetShaUra} it follows that $\Gamma(\phi)=\Gamma_{\mathrm{reg}}\cup\Gamma_{\mathrm{sing}}$, where $\Gamma_{\mathrm{reg}}$ is a relatively open subset of $\Gamma(\phi)$ with analytic regularity (\cite[Theorem~4.20]{PetShaUra}), while $x_0\in\Gamma_{\mathrm{sing}}$ if and only if
  \begin{align*}
    \lim_{r\to0^+}\frac 1r \, \mindiam \bigl( \{\phi=0\}\cap B_r(x_0) \bigr) = 0
  \end{align*}
(\cite[Proposition~7.1]{PetShaUra}), from which it also follows that the Lebesgue density of $\{\phi=0\}$ is 0 at each point of $\Gamma_{\mathrm{sing}}$.

Now the properties in the statement follow by observing that $\partial\Omega_0\subset\Gamma(\phi)$ and $\Gamma(\phi)\setminus\partial\Omega_0\subset\Gamma_{\mathrm{sing}}$: indeed, the second inclusion is a consequence of the fact that a regular point $x_0\in\Gamma_{\mathrm{reg}}$ has a neighborhood in which $\Gamma(\phi)$ is regular, which implies that $x_0\in\partial\Omega_0$.
\end{proof}

In the following proposition we show how points in the regular part $\Gamma$ or in the singular part $\Sigma$ of $\partial\Omega_0$ (\emph{regular points} and \emph{singular points}, respectively) can be characterized in terms of the blow-up of the potential $\phi$ at those points \cite{Caf80}, and a structure result for the singular part \cite{Caf98}. A different characterization can be provided in terms of the Ou-Weiss energy functional, see \cite{Ou-1994,Wei}.

\begin{proposition}[Characterization of the singular set of $\p \Ome_0$] \label{prp-singbound} %
Under the assumptions of Theorem~\ref{thm-regularity}, let $\p\Ome_0 = \Gam \cup \Sig$, where $\Gam$ is the regular part of $\p \Ome_0$ and $\Sig$ is the singular part of $\p \Ome_0$.
The sets $\Gam$ and $\Sig$ can be characterized as follows: for $x_0\in\partial\Omega_0$, the corresponding rescaled potential $\phi_{r,x_0}$ is
  \begin{align*}
    \phi_{r,x_0}(x) := \frac{\phi(x_0+rx)-\phi(x_0)}{r^2}\,.
  \end{align*}
Then, after extraction of a subsequence, we have $\phi_{r,x_0}\to\phi_{x_0}$ in $C^{1,\alpha}_{\mathrm{loc}}(\R^3)$ for every $\alpha\in(0,1)$. The blow-up function $\phi_{x_0}$ has two possible behaviors, independent of the choice of subsequence: either $\phi_{x_0}$ resembles a half space solution, i.e.
  \begin{align} \label{phi-hs} \phi_{x_0}(x)=\frac12 \bigl[(x\cdot e)^+\bigr]^2 \qqquad \text{(half-space solution)}
  \end{align}
for some unit vector $e \in S^2$, or
  \begin{align} \label{phi-ps} %
    \phi_{x_0}(x)=\frac12 x \cdot A_{x_0}x \qqquad \text{(polynomial solution)}
  \end{align}
for some symmetric matrix $A_{x_0}$ with $\trace A_{x_0}=1$.
Then $x_0\in\Gamma$ if and only if \eqref{phi-hs} holds, while $x_0\in\Sigma$ if and only if \eqref{phi-ps} holds.

Moreover, setting for $d=0,1,2$
  \begin{align*}
    \Sigma^d := \{ x_0\in\Sigma \ :\ \dim\ker A_{x_0} = d\}\,,
  \end{align*}
each set $\Sigma^d$ is contained in a countable union of $d$-dimensional $C^1$-manifolds.
Finally, $\Sigma^0=\emptyset$.
\end{proposition}

\begin{proof}
For a proof of the classification of regular and singular points in terms of the blow-up of the potential, see \cite[Theorem~3.22 and Theorem~3.23]{PetShaUra}, while for the structure of $\Sigma$ see \cite[Theorem~7.9]{PetShaUra}.

We have only to show that the set $\Sigma^0$ is actually empty. Indeed, for $x_0\in\Sigma$ one has the decay estimate
  \begin{align*}
    |\phi(x) - {\textstyle\frac12} A_{x_0}(x-x_0)\cdot(x-x_0)| \leq \sigma(|x-x_0|) \ |x-x_0|^2
  \end{align*}
where $\sigma$ is a suitable modulus of continuity (see \cite[Proposition~7.7]{PetShaUra}). This property clearly implies that if $x_0\in\Sigma^0$ we have $\phi>0$ in $B_r(x_0)\setminus\{x_0\}$ for $r>0$ small enough, which in turn yields $x_0\notin\partial\Omega_0$.
\end{proof}

%%%%%%%%%%%%%%%%%%%%%%%%%%%%%%%%%%%%%%%%%%%%%%%%%%%%%%%%%%%%%%%%%%%%%%%%%%%%%%%%%%%%%%%%%%%%%%%%%%%%%%%%%
%%%%%%%%%%%%%%%%%%%%%%%%%%%%%%%%%%%%%%%%%%%%%%%%%%%%%%%%%%%%%%%%%%%%%%%%%%%%%%%%%%%%%%%%%%%%%%%%%%%%%%%%%

\section{A surface charge model} \label{sec-surface} %

In this section we discuss the asymptotic limit, in the variational sense of $\Gamma$-convergence (see
\cite{Bra,DM}), of our charge distribution model when the charge density of one phase is much higher than the one of the
other: this is achieved by rescaling the negative charge density by a factor $\frac{1}{\e}$ and by letting $\e$ go to
zero. In the limit model the admissible configurations are described by positive Radon measures supported in
$\R^3\setminus\Omega^+$, with the optimal configuration realized by a surface distribution of charge concentrated on
$\partial\Omega^+$. We remark that a similar limit model, in the particular case where the fixed domain $\Omega^+$ is
the union of a finite number of disjoint balls, was analyzed in \cite{CapFri}.

\medskip

Given two positive Radon measures $\mu,\nu\in\mathcal{M}^+(\R^3)$, we introduce the energy
\begin{equation*}
  \mathcal{I}(\mu,\nu) := \int_{\R^3}\int_{\R^3} \frac{1}{4\pi|x-y|}\dds\mu(x)\dd\nu(y)\,,
\end{equation*}
and we set $\mathcal{I}(\mu):=\mathcal{I}(\mu,\mu)$.  We also define the potential
\begin{align} \label{eq-potmu} \phi_\mu(x):= \int_{\R^3}\frac{1}{4\pi|x-y|}\dds\mu(y)\,,
\end{align}
and we note that
\begin{equation*}
  \mathcal{I}(\mu,\nu) = \int_{\R^3}\phi_\mu(x)\dds\nu(x) = \int_{\R^3}\phi_\nu(y)\dds\mu(y)\,.
\end{equation*}
We will denote by $\mu^+:=\chi_{\Omega^+}\LL^3$ the measure associated with the uniform charge distribution in
$\Omega^+$, and by $\phi_{\mu^+}$ the associated potential.  For $\lambda>0$, $\e>0$, we define the sets
\begin{align*}
  \mathcal{A}_\lambda &:= \biggl\{ \mu\in\mathcal{M}^+(\R^3) \ :\  \supp\mu\subset\R^3\setminus\Omega^+,\  \int_{\R^3}\dd\mu \leq\lambda \biggr\}\,,\\
  \mathcal{A}_{\lambda,\e} &:= \biggl\{ \mu\in\mathcal{A}_\lambda \ :\ \mu = u\LL^3,\ u:\R^3\to\big[0,{\e}^{-1}\big] \biggr\}\ 
\end{align*}
and the functionals on $\mathcal{M}(\R^3)$
\begin{equation*}
  \mathcal{F}_\e(\mu):=
  \begin{cases}
    -2\mathcal{I}(\mu^+,\mu) + \mathcal{I}(\mu) & \text{if }\mu\in\mathcal{A}_{\lambda,\e},\\
    \infty & \text{otherwise,}
  \end{cases}
\end{equation*}
\begin{equation*}
  \mathcal{F}(\mu):=
  \begin{cases}
    -2\mathcal{I}(\mu^+,\mu) + \mathcal{I}(\mu) & \text{if }\mu\in\mathcal{A}_{\lambda},\\
    \infty & \text{otherwise.}
  \end{cases}
\end{equation*}
\begin{theorem} \label{thm-gammaconv} %
  Assume that $\Omega^+\subset\R^3$ is an open, bounded set with Lipschitz boundary.
  The family of functionals $(\mathcal{F}_\e)_\e$ $\Gamma$-converge, as $\e\to0$, to the functional $\mathcal{F}$ with
  respect to weak*-convergence in $\mathcal{M}(\R^3)$.
\end{theorem}
\begin{proof}
  We prove the two properties of the definition of $\Gamma$-convergence.

  \medskip

  \textit{Liminf inequality.}  Given $\mu_\e\wstar\mu$ weakly* in $\mathcal{M}(\R^3)$, we have to
  show that $\mathcal{F}(\mu)\leq\liminf_{\e\to0}\mathcal{F}_\e(\mu_\e)$.  We can assume without loss of generality that
  $\liminf_{\e\to0}\mathcal{F}_\e(\mu_\e)<\infty$, so that $\mu_\e\in\mathcal{A}_{\lambda,\e}$ and $\mu_\e=u_\e\LL^3$.
  Clearly $\supp\mu\subset\R^3\setminus\Omega^+$, and by lower semicontinuity
  $\mu(\R^3)\leq\liminf_{\e\to 0}\mu_\e(\R^3)\leq\lambda$, which implies that $\mu\in\mathcal{A}_\lambda$.  We then need to
  show that
  \begin{align*}
    -2\mathcal{I}(\mu^+,\mu)+\mathcal{I}(\mu) \leq\liminf_{\e\to0} \Bigl(
    -2\mathcal{I}(\mu^+,\mu_\e)+\mathcal{I}(\mu_\e) \Bigr)\,.
  \end{align*}
  Since the functional $\mathcal{I}$ is lower semicontinuous with respect to weak*-convergence of positive measures (see
  \cite[equation (1.4.4)]{Lan}), we immediately have
  \begin{align*}
    \mathcal{I}(\mu)\leq\liminf_{\e\to0}\mathcal{I}(\mu_\e)\,.
  \end{align*}
  Moreover, the convergence $\mu_\e\wstar\mu$ and $\sup_\e\mu_\e(\R^3)<\infty$ imply that
  \begin{align*}
    \lim_{\e\to0}\int_{\R^3}f\dds\mu_\e = \int_{\R^3}f\dds\mu
  \end{align*}
  for every $f\in C^0_0(\R^3):=\{g\in C^0(\R^3) : \{|g|>\e\} \text{ is compact for every }\e>0\}$.  Hence, since
  $\phi_{\mu^+}\in C^0_0(\R^3)$ we conclude that
  \begin{align} \label{eq-gammaconv1} \lim_{\e\to0}\mathcal{I}(\mu^+,\mu_\e) = \lim_{\e\to0}\int_{\R^3}\phi_{\mu^+}\
    \dd\mu_\e = \int_{\R^3}\phi_{\mu^+}\dds\mu = \mathcal{I}(\mu^+,\mu)\,,
  \end{align}
  which completes the proof of the liminf inequality.

  \medskip

  \textit{Limsup inequality.}  Given a measure $\mu\in\mathcal{M}(\R^3)$, we need to construct a
  recovery sequence $\mu_\e\wstar\mu$ such that $\limsup_{\e\to0}\mathcal{F}_\e(\mu_\e)\leq\mathcal{F}(\mu)$.  We can
  assume without loss of generality that $\mathcal{F}(\mu)<\infty$, so that $\mu\in\mathcal{A}_\lambda$.

  We first show that without loss of generality we can assume that
  $\supp\mu\subset\subset\R^3\setminus\overline{\Omega^+}$.  Indeed, since $\partial\Omega^+$ is Lipschitz, we can
  define for every $\delta>0$ a map $\Phi_\delta\in C^\infty(\R^3;\R^3)$ such that
  $\Omega^+\subset\subset\Phi_\delta(\Omega^+)$ and $\|\Phi_\delta-Id\|_{C^1(\R^3)}\to0$ as $\delta\to0$ (the map
  $\Phi_\delta$ ``pushes'' the boundary of $\Omega^+$ in the complement of $\Omega^+$).  We define the push-forward
  $\mu_\delta$ of the measure $\mu$ by setting for every continuous function $f$
  \begin{align*}
    \int_{\R^3}f\dds\mu_\delta := \int_{\R^3}f\circ\Phi_\delta\dds\mu\,.
  \end{align*}
  It is not hard to see that $\mu_\delta\in\mathcal{A}_\lambda$,
  $\supp\mu_\delta\subset\subset\R^3\setminus\overline{\Omega^+}$, $\mu_\delta\wstar\mu$ weakly* in $\mathcal{M}(\R^3)$
  and $\mathcal{F}(\mu_\delta)\to\mathcal{F}(\mu)$ as $\delta\to0$.  This shows that it is sufficient to provide a
  recovery sequence in the case $\supp\mu\subset\subset\R^3\setminus\overline{\Omega^+}$.

  \medskip

  We now reduce to the case of a measure absolutely continuous with respect to the Lebesgue measure.  Indeed, we define
  for $\delta>0$ the convolution
  \begin{align*}
    \mu_\delta := \rho_\delta*\mu = \int_{\R^3}\rho_\delta(\cdot - y)\dds\mu(y)\,,
  \end{align*}
  where $\rho_\delta\in C^{\infty}_{\mathrm c}(B_\delta)$, $\rho_\delta\geq0$, $\int_{B_\delta}\rho_\delta=1$ is a
  sequence of mollifiers.  Then $\mu_\delta\in\mathcal{A}_\lambda$ for $\delta$ sufficiently small (since we are
  assuming $\supp\mu\subset\subset\R^3\setminus\overline{\Omega^+}$), $\mu_\delta$ is absolutely continuous with respect
  to the Lebesgue measure and $\mu_\delta\wstar\mu$ (see \cite[Theorem~1.26]{Mat95}).  We now show that we also have
  $\mathcal{F}(\mu_\delta)\to\mathcal{F}(\mu)$.  Indeed, the convergence of $\mathcal{I}(\mu^+,\mu_\delta)$ to
  $\mathcal{I}(\mu^+,\mu)$ can be proved exactly as in \eqref{eq-gammaconv1}; moreover
  \begin{align*}
    \mathcal{I}(\mu) = \int_{B_\delta}\rho_\delta(z)\mathcal{I}(\mu(\cdot - z))\dds z \geq
    \mathcal{I}\biggl(\int_{B_\delta}\rho_\delta(z)\dds\mu(\cdot-z)\biggr) = \mathcal{I}(\mu_\delta)
  \end{align*}
  (the first equality is due to the translation invariance of the functional $\mathcal{I}$, while the inequality is a
  consequence of Jensen's inequality and of the convexity of $\mathcal{I}$), which combined with the lower
  semicontinuity of $\mathcal{I}$ leads to $\lim_{\delta\to0}\mathcal{I}(\mu_\delta)=\mathcal{I}(\mu)$.  This yields
  $\mathcal{F}(\mu_\delta)\to\mathcal{F}(\mu)$.
\medskip

  Hence, to complete the proof it remains just to provide a recovery sequence in the case of a measure
  $\mu\in\mathcal{A}_\lambda$ absolutely continuous with respect to the Lebesgue measure and such that
  $\supp\mu\subset\subset\R^3\setminus\overline{\Omega^+}$.  This can be done by a simple truncation argument: denoting
  by $u$ the Lebesgue density of $\mu$, we define $\mu_\e:=(u\wedge\frac{1}{\e})\mathcal{L}^3$: it is then clear that
  $\mu_\e\in\mathcal{A}_{\lambda,\e}$ and that $\mathcal{F}_\e(\mu_\e)\to\mathcal{F}(\mu)$ by the Lebesgue Dominated
  Convergence Theorem.
\end{proof}
In the following proposition we discuss the limit problem, showing that the minimizer is obtained by a surface
distribution of charge on $\partial\Omega^+$.
\begin{proposition}\label{prp-surfacecharge}
  Let $\mu\in\mathcal{M}^+(\partial\Omega^+)$ be a solution to the minimum problem
  \begin{align} \label{min3} \min\biggl\{ \mathcal{F}(\mu) : \mu\in\mathcal{M}^+(\partial\Omega^+),\
    \int_{\partial\Omega^+}\dds\mu=m \biggr\}\,.
  \end{align}
  Then $\mu$ is the unique minimizer of $\mathcal{F}$ over $\mathcal{A}_m$.
\end{proposition}
\begin{proof}
  We start by observing that the existence of a minimizer is guaranteed by the direct method of the Calculus of
  Variations. Indeed, given a minimizing sequence $(\mu_n)_n$, by the uniform bound $\mu_n(\partial\Omega^+)=m$ we can
  extract a (not relabeled) subsequence weakly*-converging to some positive measure $\mu$ supported on
  $\partial\Omega^+$, and such that $\mu(\partial\Omega^+)=m$. Moreover, by semicontinuity of $\mathcal{I}$ with respect
  to weak*-convergence of positive measures, and by the same argument as in \eqref{eq-gammaconv1} with $\mu_\e$ replaced
  by $\mu_n$, we easily obtain that $\mu$ is a minimizer of \eqref{min3}.

  \medskip

  We now claim that the potential $\phi_\mu$ generated by the minimizer $\mu$, according to \eqref{eq-potmu}, coincides
  with the potential $\phi_{\mu^+}$ outside $\Omega^+$.  Fix any point $x_0\in\partial\Omega^+$, let $\rho>0$ and denote
  $\alpha_\rho:=\mathcal{H}^2(\partial\Omega^+\cap B_\rho(x_0))$.  Consider for $\e>0$ the measure
  \begin{align*}
    \mu_\e := \e\mathcal{H}^2\Mattes(\partial\Omega^+\cap B_\rho(x_0)) + \frac{m-\e\alpha_\rho}{m}\ \mu\,,
  \end{align*}
  which is admissible in problem \eqref{min3} for $\e$ sufficiently small.  Then, by minimality of $\mu$
  \begin{align*}
    0 &\geq -2\mathcal{I}(\mu^+,\mu) + \mathcal{I}(\mu) + 2\mathcal{I}(\mu^+,\mu_\e) - \mathcal{I}(\mu_\e)\\
    & = \frac{2\e\alpha_\rho}{m} \int_{\partial\Omega^+}\int_{\partial\Omega^+}\frac{1}{4\pi|x-y|}\dds\mu(x)\dd\mu(y)
    - \frac{2\e\alpha_\rho}{m} \int_{\Omega^+}\int_{\partial\Omega^+}\frac{1}{4\pi|x-y|}\dds x\dds\mu(y) \\
    & \qquad + 2\e\int_{\Omega^+}\int_{\partial\Omega^+\cap B_\rho(x_0)}\frac{1}{4\pi|x-y|}\dds x\dds\mathcal{H}^2(y) \\
    & \qquad - 2\e\int_{\partial\Omega^+}\int_{\partial\Omega^+\cap B_\rho(x_0)}\frac{1}{4\pi|x-y|}\
    \dd\mathcal{H}^2(x)\dd\mu(y) + o(\e)\,.
  \end{align*}
  Dividing by $\e$ and letting $\e\to0^+$ we obtain
  \begin{align*}
    \frac{1}{m} \int_{\partial\Omega^+} (\phi_{\mu^+}-\phi_\mu)\dds\mu \geq
    \frac{1}{\alpha_\rho}\int_{\partial\Omega^+\cap B_\rho(x_0)}(\phi_{\mu^+}-\phi_\mu)\dds\mathcal{H}^2\,,
  \end{align*}
  and since the right-hand side in the previous inequality converges as $\rho\to0$ to $(\phi_{\mu^+}-\phi_\mu)(x_0)$, we
  obtain that
  \begin{align*}
    \frac{1}{\mu(\partial\Omega^+)}\int_{\partial\Omega^+}(\phi_{\mu^+}-\phi_\mu)\dds\mu \geq
    (\phi_{\mu^+}-\phi_\mu)(x_0)
  \end{align*}
  for every $x_0\in\partial\Omega^+$.  We then conclude that there exists a constant $\alpha$ such that $\phi :=
  \phi_{\mu^+}-\phi_\mu = \alpha$ $\mu$-a.e. on $\partial\Omega^+$, and $\phi\leq\alpha$ on $\partial\Omega^+$.

  \medskip

  Observe that, if $R_0>0$ denotes a radius such that $\Omega^+\subset B_{R_0}$, by \eqref{eq-pot2} (which still holds
  in the present setting: see, for instance, \cite[Theorem~6.12]{Hel}) we have
  \begin{align} \label{eq-gammaconv4} \int_{\partial B_R}\phi\dds\mathcal{H}^2 = 0 \qquad\text{for every }R>R_0\ .
  \end{align}
  Now, since $\phi$ is superharmonic in $\R^3\setminus\supp\mu$, $\phi=\alpha$ on $\supp\mu$ and $\phi$ vanishes at
  infinity, by the minimum principle we have that
  \begin{align}\label{eq-gammaconv5}
    \phi\geq\min\{0,\alpha\} \qquad\text{in }\R^3\setminus\supp\mu\ .
  \end{align}
  Hence condition \eqref{eq-gammaconv4} excludes the case $\alpha>0$.  On the other hand, if $\alpha<0$ then we would
  have that $\phi$ is harmonic in $\R^3\setminus\overline{\Omega^+}$, $\phi\leq\alpha<0$ on $\partial\Omega^+$ and
  $\phi$ vanishes at infinity, so that $\phi<0$ in $\R^3\setminus\overline{\Omega^+}$, which is again a contradiction
  with \eqref{eq-gammaconv4}.  Thus $\alpha=0$ and combining \eqref{eq-gammaconv5} with the fact that $\phi\leq\alpha$
  on $\partial\Omega^+$, we conclude that $\phi=0$ on $\partial\Omega^+$.  In turn, this implies that $\phi=0$ in
  $\R^3\setminus\Omega^+$ since $\phi$ is harmonic in $\R^3\setminus\overline{\Omega}^+$ and vanishes at infinity.  We
  have then proved that
  \begin{align} \label{eq-gammaconv2} \int_{\partial\Omega^+}\frac{1}{4\pi|x-y|}\dds\mu(y) =
    \int_{\Omega^+}\frac{1}{4\pi|x-y|}\dds y \qquad\text{for every }x\in\R^3\setminus\Omega^+\ .
  \end{align}

  \medskip

  We can now complete the proof of the proposition, showing that $\mu$ is the minimizer of $\mathcal{F}$ over
  $\mathcal{A}_m$.  Indeed, for every $\nu\in\mathcal{A}_m$ we have, using \eqref{eq-gammaconv2}
  \begin{align*}
    \mathcal{F}(\nu) & = -2\int_{\R^3}\int_{\Omega^+}\frac{1}{4\pi|x-y|}\dds x\dds\nu(y)
    + \int_{\R^3}\int_{\R^3}\frac{1}{4\pi|x-y|}\dds\nu(x)\dd\nu(y) \\
    & = -2\int_{\R^3}\int_{\partial\Omega^+}\frac{1}{4\pi|x-y|}\dds\mu(x)\dd\nu(y)
    + \int_{\R^3}\int_{\R^3}\frac{1}{4\pi|x-y|}\dds\nu(x)\dd\nu(y) \\
    & = \int_{\R^3}\int_{\R^3}\frac{1}{4\pi|x-y|}\dds(\mu-\nu)(x)\dd(\mu-\nu)(y)
    - \int_{\R^3}\int_{\R^3}\frac{1}{4\pi|x-y|}\dds\mu(x)\dd\mu(y) \\
    & = \mathcal{I}(\mu-\nu) + \mathcal{F}(\mu)\,.
  \end{align*}
  Using the fact that $\mathcal{I}(\mu-\nu)\geq0$, with equality if and only if $\mu=\nu$ (see
  \cite[Theorem~1.15]{Lan}), we obtain the conclusion.
\end{proof}
\begin{remark} \label{rem-surfacepot} The proof of the previous proposition shows, in particular, the following
  interesting fact: if $\mu$ solves the minimum problem \eqref{min3}, then
  \begin{align*}
    \int_{\partial\Omega^+}\frac{1}{4\pi|x-y|}\dds\mu(y) = \int_{\Omega^+}\frac{1}{4\pi|x-y|}\dds y \qquad\text{for
      every }x\in\R^3\setminus\Omega^+,
  \end{align*}
  that is, the potential $\phi_\mu$ generated by $\mu$ coincides outside of $\Omega^+$ with the potential $\phi_{\mu^+}$
  generated by the uniform distribution of charge in $\Omega^+$. In particular, we again find a complete screening
  property: the potential of $\mu^+ -\mu$ vanishes outside of the support of that measure.
\end{remark}

%%%%%%%%%%%%%%%%%%%%%%%%%%%%%%%%%%%%%%%%%%%%%%%%%%%%%%%%%%%%%%%%%%%%%%%%%%%%%%%%%%%%%%%%%%%%%%%%%%%%%%%%%
%%%%%%%%%%%%%%%%%%%%%%%%%%%%%%%%%%%%%%%%%%%%%%%%%%%%%%%%%%%%%%%%%%%%%%%%%%%%%%%%%%%%%%%%%%%%%%%%%%%%%%%%%

\appendix
\section{Spherically symmetric configurations} \label{sec-example} %

An example in which the shape of the minimizer of \eqref{min-unc} can be explicitly determined is when $\Omega^+$ has spherical symmetry: if $\Omega^+$ is an annulus, the minimizer is given by the union of two annuli touching $\Omega^+$ from the interior and from the exterior respectively (Proposition~\ref{prp-annuli}); in the particular case of a ball, the corresponding minimizer is an annulus around $\Omega^+$ (Corollary~\ref{cor-ball}).

In proving this result we will also compute explicitly the general formula for the energy of balls and annuli. We remark that similar expressions were computed in \cite{GenPel} for spherically symmetric monolayers and bilayers, where the two phases are adjacent and enclose the same volume in any dimensions, by deriving the explicit value of the associated potential.

In the following, given $r_2\geq r_1\geq0$ we denote by $C_{r_1,r_2}:=B_{r_2}\setminus\overline{B}_{r_1}$ the open
annulus enclosed by the radii $r_1$, $r_2$.

\begin{proposition} \label{prp-annuli} %
  Let $\Omega^+=C_{R_1,R_2}$ for some $R_2>R_1>0$, and let $R_*>1$ be the unique solution of $2(R_*^2-1)-(2(R_*^3-1))^{2/3}=
  0$. Let $\Ome^-$ be the minimizer of \eqref{min-unc}. Then:
  \begin{enumerate}
  \item If $\frac{R_2}{R_1}< R_*$ then $\Omega^- = C_{r_1,R_1}\cup C_{R_2,r_2}$, where $r_1\in(0,R_1)$ and $r_2>R_2$ are determined by
    \begin{align}
      \begin{cases}
        r_1^3-r_2^3+2(R_2^3-R_1^3)=0,\\
        r_1^2-r_2^2+2(R_2^2-R_1^2)=0.
      \end{cases}
      \label{eq-opt-bilayer}
    \end{align}
  \item If $\frac{R_2}{R_1} \geq R_*$ then $\Omega^-=B_{R_1}\cup C_{R_2,r}$ with $r=\sqrt[3]{2(R_2^3-R_1^3)}$.
  \end{enumerate}
\end{proposition}

\begin{proof}
We divide the proof into two steps.

\medskip

{\it Step 1.} We start by computing the self-interaction energy of an annulus $C_{r_1,r_2}$ in spherical coordinates:
\begin{equation*}
  \int_{C_{r_1,r_2}} \int_{C_{r_1,r_2}} \frac{1}{4\pi|x-y|}\dds x \dd y
  = \int_{r_1}^{r_2}\dd\rho \int_0^{2\pi}\dd\theta \int_0^\pi \dd\phi
  \biggl( \int_{C_{r_1,r_2}} \frac{\rho^2\sin\phi}{4\pi|x(\rho,\theta,\phi)-y|}\dds y \biggr)
\end{equation*}
where $x(\rho,\theta,\phi)$ denotes the point in $C_{r_1,r_2}$ whose spherical coordinates are
$(\rho,\theta,\phi)$. Since by rotation invariance the inner integral depends in fact only on $\rho$, we can compute it
for $x=\rho e_3$, obtaining
\begin{align*}
  \hspace{6ex} & \hspace{-6ex} %
  \int_{C_{r_1,r_2}} \int_{C_{r_1,r_2}} \frac{1}{4\pi|x-y|}\dds x \dd y
  = \int_{r_1}^{r_2} \rho^2 \biggl( \int_{C_{r_1,r_2}} \frac{\dd y}{|\rho e_3-y|} \biggr) \dd\rho\\
  &= \int_{r_1}^{r_2} \rho^2 \biggl( \int_{r_1}^{r_2} \int_0^{2\pi} \int_0^\pi \frac{r^2\sin\phi}{(r^2+\rho^2-2r\rho\cos\phi)^{1/2}} \dds\phi \dd\theta \dd r  \biggr) \dd\rho\\
  &= 2\pi \int_{r_1}^{r_2} \int_{r_1}^{r_2} \rho r \bigl(r+\rho-|r-\rho|\bigr) \dds\rho\dd r\,.
\end{align*}
From the explicit computation of the last integral we conclude that
\begin{align}
\int_{C_{r_1,r_2}} \int_{C_{r_1,r_2}} \frac{1}{4\pi|x-y|}\dds x \dd y =
  \frac{4\pi}{15}\Bigl(3r_1^5+2r_2^5-5r_1^3r_2^2\Bigr)\,. \label{eq:ring}
\end{align}
By similar computations we obtain the interaction energy of two disjoint annuli: for $r_2\geq r_1\geq R_2\geq R_1>0$ we have
\begin{equation*}
  \int_{C_{r_1,r_2}} \int_{C_{R_1,R_2}} \frac{1}{4\pi|x-y|}\dds x \dd y
  = \frac{2\pi}{3} \bigl(R_2^3-R_1^3\bigr) \bigl( r_2^2 - r_1^2 \bigr) \,.
\end{equation*}

\medskip

{\it Step 2.} We now turn to the proof of the statement. Notice that by scaling we can assume without loss of generality that the inner radius of the annulus is equal to 1, and in particular we can consider $\Omega^+=C_{1,R}$ for some $R>1$. By uniqueness, $\Ome^-$ is invariant under rotations, so that it consists of a union of annuli. By Theorem~\ref{thm-support} $\Ome^-$ is the union of an annulus, touching $\Omega^+$ from the exterior, and an annulus or a ball touching $\Omega^+$ from the interior.

  The first possibility is that the connected component of $\Omega^-$ internal to $\Omega^+$ is not the unit ball; the
  minimizer is then the union of two annuli, $\Omega^-=C_{r_1,1}\cup C_{R,r_2}$. The conditions \eqref{eq-opt-bilayer}
  are induced by the volume constraint and by choosing $r_1>0,r_2>1$ such that the energy, computed in the previous step, has a local minimum (or, equivalently that $\phi$ vanishes on $\partial B_{r_1}$ and $\partial B_{r_2}$). Since by \eqref{eq-opt-bilayer}
  \begin{align*}
    2(R^2-1) \ =\ r_2^2 -r_1^2 \ \leq\ \big(r_2^3-r_1^3\big)^{\frac{2}{3}} \ =\ \big(2(R^3-1)\big)^{\frac{2}{3}}
  \end{align*}
  we conclude that this possibility can only occur for $2(R^2-1) -\big(2(R^3-1)\big)^{\frac{2}{3}}\leq 0$, which is
  equivalent to $R\leq R_*$.

  The second possibility is that the connected component of $\Omega^-$ internal to $\Omega^+$ is the unit ball, with an
  external annulus $C_{R,r}$, where $r$ is determined by the volume constraint. After some calculations we obtain in
  this case that
  \begin{align*}
    \phi(0) \ =\ (R^2-1) -\frac{1}{2}\big(2(R^3-1)\big)^{\frac{2}{3}}\,.
  \end{align*}
  Since for any minimizer $\phi\geq 0$ holds by Lemma~\ref{lem-phipos} we obtain that the second possibility can only
  occur if $2(R^2-1) -\big(2(R^3-1)\big)^{\frac{2}{3}}\geq 0$, which is equivalent to $R\geq R_*$ (in case of equality
  we just have $r_1=0$ in the first case, that is the inner annulus degenerates to the unit-ball). Since we always are
  in one of the two cases the claim follows.
\end{proof}

\begin{corollary}[The case of a ball] \label{cor-ball} %
  Let $\Omega^+=B_R$ for some $R>0$. Then the minimizer of \eqref{min-unc} is the annulus $\Omega^-=C_{1,\sqrt[3]{2}R}$.
\end{corollary}

\begin{proof}
  The result follows immediately from Proposition~\ref{prp-annuli} by letting $R_1\to0$. In this case the argument is actually simpler and we can give a short independent proof: by Theorem~\ref{thm-unconstrained}, the minimizer is unique up to a set of vanishing Lebesgue measure. Since $\Ome^+$ is invariant with respect to rotations, it hence follows that also $\Ome^-$ is rotationally invariant and hence consists of an union of annuli. The annulus $\Omega^-=C_{1,\sqrt[3]{2}R}$ is the only set with this property such that Theorem~\ref{thm-support} holds for every connected component $V$ of $\Ome^-$ and satisfying the charge neutrality condition $|\Omega^-|=|\Omega^+|$.
\end{proof}

\begin{figure}
\begin{tikzpicture}[scale=1]
\draw [fill=blue!15!white] (1.259921,0) circle [radius=1.259921];
\draw [fill=blue!15!white] (-1.259921,0) circle [radius=1.259921];
\draw [fill=blue!70!white] (1.259921,0) circle [radius=1];
\draw [fill=blue!70!white] (-1.259921,0) circle [radius=1];
\draw [fill] (0,0) circle [radius=0.05];
\node at (0,-1.3) {$\Omega^+$};
\draw (0,-1) to (-1.259921,0);
\draw (0,-1) to (1.259921,0);
\node at (0,1.3) {$\Omega^-$};
\draw (0,1) to (0.5,0.8);
\draw (0,1) to (-0.5,0.8);
%\node at (1.3,0) {$\Omega^+$};
%\node at (-1.3,0) {$\Omega^+$};
\end{tikzpicture}
\caption{Example of occurrence of a singularity at the origin in the boundary of a minimizer.}
\label{fig-sing}
\end{figure}
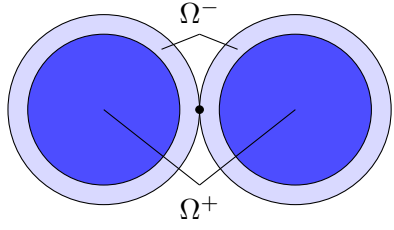
\begin{remark} \label{rem-singularity}
We can construct a simple example of occurrence of a singularity in the boundary of a minimizer as follows. Let $\Omega^+:= B_1(\sqrt[3]{2}e_1)\cup B_1(-\sqrt[3]{2}e_1)$ be the union of two disjoint unit balls centered at points at distance $2\sqrt[3]{2}$. The corresponding minimizing configuration is then the union of the two annuli which minimize separately the energy for the two connected components of $\Omega^+$, i.e. $\Omega^-=C_{1,\sqrt[3]{2}}(\sqrt[3]{2}e_1)\cup C_{1,\sqrt[3]{2}}(-\sqrt[3]{2}e_1)$. Indeed, by linearity the resulting potential is nonnegative and vanishes outside $\overline{\Omega^+}\cup\Omega^-$, and by uniqueness of the configuration having the screening property (see Remark~\ref{rem-screening}) we conclude that $\Omega^-$ is the minimizer. Notice that the two annuli touch at the origin, which is thus a singular point (see Figure~\ref{fig-sing}).
\end{remark}

%%%%%%%%%%%%%%%%%%%%%%%%%%%%%%%%%%%%%%%%%%%%%%%%%%%%%%%%%%%%%%%%%%%%%%%%%%%%%%%%%%%%%%%%%%%%%%%%%%%%%%%%%
%%%%%%%%%%%%%%%%%%%%%%%%%%%%%%%%%%%%%%%%%%%%%%%%%%%%%%%%%%%%%%%%%%%%%%%%%%%%%%%%%%%%%%%%%%%%%%%%%%%%%%%%%

\bigskip \bigskip
\noindent {\bf Acknowledgments.}
MB is member of the INdAM - GNAMPA Project 2015 ``Critical Phenomena in the Mechanics of Materials: a Variational Approach''. MB and HK want to thank the Mathematics Center Heidelberg (MATCH) for support. HK would like to thank S.~M\"uller for interesting discussions.

%%%%%%%%%%%%%%%%%%%%%%%%%%%%%%%%%%%%%%%%%%%%%%%%%%%%%%%%%%%%%%%%%%%%%%%%%%%%%%%%%%%%%%%%%%%%%%%%%%%%%%%%%

\end{document}